\documentclass{article}%
\usepackage{amsmath}
\usepackage{amsfonts}
\usepackage{amssymb}
\usepackage{graphicx}

\usepackage{color}%
\usepackage{esint}
\usepackage{mathtools}
\usepackage[title,titletoc]{appendix}
\providecommand{\U}[1]{\protect\rule{.1in}{.1in}}
\providecommand{\U}[1]{\protect\rule{.1in}{.1in}}
\newtheorem{theorem}{Theorem}

\newtheorem{claim}[theorem]{Claim}

\newtheorem{corollary}[theorem]{Corollary}

\newtheorem{lemma}[theorem]{Lemma}

\newtheorem{openquestion}[theorem]{Open Question}
\newtheorem{proposition}[theorem]{Proposition}
\newtheorem{remark}[theorem]{Remark}

\newenvironment{proof}[1][Proof]{\noindent\textbf{#1} }{\ \rule{0.5em}{0.5em}}

\begin{document}
 \author{Swarnendu Sil\thanks{swarnendu.sil@fim.math.ethz.ch}}
\title{Nonlinear Stein theorem for differential forms}
\maketitle 
\tableofcontents 
\begin{abstract}
	We prove that if  $u$ is an $\mathbb{R}^{N}$-valued $W^{1,p}_{loc}$ differential $k$-form  with $\delta \left( a(x) \lvert du \rvert^{p-2} du \right) \in L^{(n,1)}_{loc}$ in a 
	domain of $\mathbb{R}^{n}$ for $N \geq 1,$ $n \geq 2,$ $0 \leq k \leq n-1, $ $1 < p < \infty, $ with uniformly positive, bounded, Dini continuous scalar function $a$, 
	then $du$ is continuous. This generalizes the classical result by Stein in the scalar case and the work of Kuusi-Mingione for the $p$-Laplacian type systems. We also discuss H\"{o}lder, BMO and VMO regularity estimates for such systems when $p \geq 2.$     
\end{abstract}
\section{Introduction}
In the present article we are concerned with the regularity estimates for the inhomogenous quasilinear system 
\begin{align}\label{quasilinear maxwell tangential intro}
\delta ( a(x) \lvert du \rvert^{p-2} du) )  = f   &&\text{ in } \Omega,  
\end{align}
where $a$ is a uniformly positive, bounded, continuous function and $u, f$ are $\mathbb{R}^{N}$-valued $k$-differential forms defined on an open, bounded, smooth subset $\Omega \subset \mathbb{R}^{n},$ with $n \geq 2,$ 
$N \geq 1,$  $0 \leq k \leq n-1$ and $1 <p < \infty.$ 

\paragraph*{} If $k=0$,  this is the well-known inhomogenous $p$-Laplacian ( with coefficients ) equation or system, for $N=1$ or $N > 1,$ respectively. Both has been studied quite extensively and can be justifiably called the prototypical operators for the study of quasilinear elliptic equations and systems respectively. For the homogenous case, the pioneering work in this field  is Uhlenbeck \cite{uhlenbecknonlinearelliptic}, who considered the very general setting of elliptic complexes. In another fundamental work Hamburger \cite{hamburgerregularity} considered the homogenous problem in precisely the present setting. However, apart from Beck-Stroffolini \cite{Beck_Stroffolini_RegularityDifferentialforms} who considered partial regularity questions for more general quasilinear systems for forms, the inhomogeneous problem for the cases $1 \leq k \leq n-1$ has received surprisingly little attention. We also remark that this level of generality is not a futile exercise, as this implies new results even in the simple but important case of vector fields in three dimensions.      

\paragraph*{Lack of ellipticity}\  The cases $1 \leq k \leq n-1$ has their own unique features. To boot, in striking contrast to the cases of the $p$-Laplcian equation and system, for $1 \leq k \leq n-1,$ the system \eqref{quasilinear maxwell tangential intro} is \emph{not} elliptic. There is no uniqueness of solutions. Indeed, adding any closed form to a solution yields another solution and thus the system has an infinite dimensional kernel. However, since adding a closed form to $u$ obviously has no impact on $du,$ one can certainly hope to prove regularity properties for the exterior derivative $du.$ 

\paragraph*{Gauge freedom and gauge fixing}\ This lack of ellipticity is due to the so-called `gauge freedom' of the problem and can be circumvented in some cases by choosing to `fix the gauge'. More precisely, we consider the system 
\begin{align}\label{gauge fixed system}
\delta \left( a(x) \lvert d u \rvert^{p-2} d u \right) = f \qquad \text{ and } \qquad \delta u = 0 \qquad \qquad \text{ in } \Omega.
\end{align}
The condition $\delta u = 0$ is called the Coulomb gauge condition and this makes the system at least formally elliptic. The crucial point is that if we are interested only in the regularity properties of $du,$ we can always assume the Coulomb condition. Indeed, given any $W^{1,p}$ local solution of \eqref{quasilinear maxwell tangential intro}, we can always find a solution of \eqref{gauge fixed system} which has the same exterior derivative a.e. and thus the lack of ellipticity is in some sense, superficial.

\paragraph*{Stein theorem} \ Stein \cite{Stein_steintheorem} proved the borderline Sobolev embedding result which states that for $n \geq 2,$ $u \in L^{1}(\mathbb{R}^{n})$ and $\nabla u \in L^{(n,1)}(\mathbb{R}^{n}; \mathbb{R}^{n})$ implies $u$ is continuous. Coupled with standard Calderon-Zygmund estimates, which extend to Lorentz spaces, this implies $u \in C^{1}(\mathbb{R}^{n})$ if $\Delta u \in L^{(n,1)}(\mathbb{R}^{n}).$ The search for a nonlinear generalization of this result culminated in 
Kuusi-Mingione \cite{KuusiMingione_nonlinearStein}, where the authors proved the following general \emph{quasilinear} \emph{vectorial} version
\begin{align}\label{stein pde version p laplacian}
\operatorname*{div}  \left( a(x)\lvert \nabla u \rvert^{p-2} \nabla u \right) \in L^{(n,1)}(\mathbb{R}^{n}; \mathbb{R}^{N}) \qquad \Rightarrow  \qquad u \in C^{1}(\mathbb{R}^{n};\mathbb{R}^{N}), 
\end{align}
where $a$ is a uniformly positive Dini continuous function and  $1 < p < \infty .$ 

\paragraph*{Stein theorem for forms}\ The main point of the present article is that by working with a different space and different boundary value problems for comparison, one can adapt the techniques in Kuusi-Mingione \cite{KuusiMingione_nonlinearStein} to our setting.
The main technical obstacle to adapt their argument in our case is that since the nonlinear information is concerned only with the exterior derivative, the natural space is not $W^{1,p}$ and we do not have Sobolev-Poincar\'{e} inequalities. So we need to use different boundary value problems to obtain the comparison estimates and use a `gauge fixing' to cater for the lack of ellipticity. The gauge fixing also allows us to transfer the some regularity information $du$ to $\nabla u$. 
Our main result in the present article is 
\begin{theorem}[Nonlinear Stein theorem]\label{main theorem}
	Let $n \geq 2 ,$ $N \geq 1$ and $0 \leq k \leq n-1$ be integers and let $\Omega \subset \mathbb{R}^{n}$ be open. Suppose that 
	\begin{itemize}
		\item[(i)] $f: \Omega \rightarrow \Lambda^{k}\mathbb{R}^{n}\otimes \mathbb{R}^{N}$ is $L^{(n,1)}$ locally in $\Omega$ and $\delta f = 0$ in $\Omega$ in the sense of distributions, 
		\item[(ii)] $a:\Omega \rightarrow [\gamma, L]$, where $ 0 < \gamma < L < \infty,$ is Dini continuous.
	\end{itemize}
	Let $1 < p < \infty$ and $u \in 
	W_{loc}^{1,p} \left(\Omega; \Lambda^{k}\mathbb{R}^{n}\otimes \mathbb{R}^{N} \right)$ be a local weak solution to the system 
	\begin{align}\label{quasilinear maxwell tangential}
	\delta ( a(x) \lvert du \rvert^{p-2} du) )  = f   &&\text{ in } \Omega. 
	\end{align}
	Then $du $ is continuous in $\Omega.$ Moreover, if in addition $\delta u = 0 $ in $\Omega,$ then $\nabla u$ is locally VMO in $\Omega.$
\end{theorem}
Note that the case $k=0$ is somewhat special where the theorem reduces to Theorem 1 of \cite{KuusiMingione_nonlinearStein} and concludes that $\nabla u$ is continuous, as in that case $du$ and $\nabla u$ is the same. 
\paragraph*{H\"{o}lder and BMO estimates for the gradient for $p \geq 2$ case}\ For $ p \geq 2,$ we also derive estimates in Campanato spaces for the gradients, which implies the following. 
\begin{theorem}\label{holder theorem}
	Let $n \geq 2 ,$ $N \geq 1$ and $0 \leq k \leq n-1$ be integers and let $2 \leq p < \infty$ and $0 < \alpha < 1$ be real numbers. Let $\Omega \subset \mathbb{R}^{n}$ be open and assume that $a:\Omega \rightarrow [\gamma, L]$, where $ 0 < \gamma < L < \infty,$ is $C^{0,\alpha}\left( \Omega \right)$. 
	Let $u \in 
	W_{loc}^{1,p} \left(\Omega; \Lambda^{k}\mathbb{R}^{n}\otimes \mathbb{R}^{N} \right)$ be a local weak solution to the system 
	\begin{align}\label{main equation holder}
	\delta ( a(x) \lvert du \rvert^{p-2} du) )  = f   &&\text{ in } \Omega. 
	\end{align}
	Then then following holds true. 
	\begin{itemize}
		\item[(i)] If $f \in  L_{loc}^{q} \left( \Omega; \Lambda^{k}\mathbb{R}^{n}\otimes \mathbb{R}^{N} \right) $ for some $q > n$ and $\delta f = 0$ in $\Omega$ in the sense of distributions, then 
		$u \in 
		C_{loc}^{1,\theta} \left(\Omega; \Lambda^{k}\mathbb{R}^{n}\otimes \mathbb{R}^{N} \right)$ for some $0 < \theta < 1,$ modulo the addition of a closed form. 
		\item[(ii)] If $f \in  L_{loc}^{n} \left( \Omega; \Lambda^{k}\mathbb{R}^{n}\otimes \mathbb{R}^{N} \right) $ and $\delta f = 0$ in $\Omega$ in the sense of distributions, then 
		$\nabla u $ is locally VMO in $\Omega$, modulo the addition of a closed form. 
		\item[(iii)] If $f \in  L_{loc}^{(n,\infty)} \left( \Omega; \Lambda^{k}\mathbb{R}^{n}\otimes \mathbb{R}^{N} \right) $ and $\delta f = 0$ in $\Omega$ in the sense of distributions, then 
		$\nabla u$ is locally BMO in $\Omega$, modulo the addition of a closed form.
	\end{itemize}
	\end{theorem} 
These results generalize the result of DiBenedetto-Manfredi \cite{DiBenedettoManfredi} to the case of differential forms. 

Restricted to the case $k=1, N=1$ and $n=3,$ theorem \ref{main theorem} and theorem \ref{holder theorem} reduce to the following corollaries, which we mention separately due to its importance in connection to nonlinear Maxwell operators and quasilinear Stokes-type problems. 
\begin{corollary}
	Let  $\Omega \subset \mathbb{R}^{3}$ be open and let $1 <  p < \infty.$ Suppose $a:\Omega \rightarrow [\gamma, L]$ is Dini continuous, where $ 0 < \gamma < L < \infty .$ Let  
	$f \in L^{(3,1)}_{loc}(\Omega, \mathbb{R}^{3})$ with $\operatorname*{div} f = 0$ in $\Omega$ in the sense of distributions. If 
	$u \in W_{loc}^{1,p} \left(\Omega; \mathbb{R}^{3} \right)$ be a local weak solution to the system 
	\begin{align}\label{quasilinear curl curl}
	\operatorname*{curl} ( a(x) \lvert \operatorname*{curl} u \rvert^{p-2} \operatorname*{curl} u )  = f   &&\text{ in } \Omega. \end{align}
	Then $\operatorname*{curl} u$ is continuous in $ \Omega.$ Moreover, if in addition, $u$ is divergence-free, then $\nabla u$ is VMO locally in $\Omega.$
\end{corollary} 
\begin{corollary}
	Let  $\Omega \subset \mathbb{R}^{3}$ be open and let $2 \leq  p < \infty.$ Suppose $a:\Omega \rightarrow [\gamma, L]$ is $C^{0,\alpha}\left( \Omega \right)$, where $ 0 < \gamma < L < \infty $ and $ 0 < \alpha < 1.$ Let  
	$u \in W_{loc}^{1,p} \left(\Omega; \mathbb{R}^{3} \right)$ be a local weak solution to the system 
	\begin{align}\label{quasilinear curl curl holder}
	\operatorname*{curl} ( a(x) \lvert \operatorname*{curl} u \rvert^{p-2} \operatorname*{curl} u )  = f   &&\text{ in } \Omega. \end{align} 
	Then then following holds true. 
	\begin{itemize}
		\item[(i)] If $f \in  L_{loc}^{q} \left( \Omega; \mathbb{R}^{3} \right) $ for some $q > 3$ and $\operatorname{div} f = 0$ in $\Omega$ in the sense of distributions, then 
		$u \in 
		C_{loc}^{1,\theta} \left(\Omega; \mathbb{R}^{3} \right)$ for some $0 < \theta < 1,$ modulo the addition of a $\operatorname{curl}$-free field. 
		\item[(ii)] If $f \in  L_{loc}^{3} \left( \Omega; \mathbb{R}^{3}\right) $ and $\operatorname{div} f = 0$ in $\Omega$ in the sense of distributions, then 
		$\nabla u $ is locally VMO in $\Omega$, modulo the addition of a $\operatorname{curl}$-free field. 
		\item[(iii)] If $f \in  L_{loc}^{(3,\infty)} \left( \Omega; \mathbb{R}^{3} \right) $ and $\operatorname{div} f = 0$ in $\Omega$ in the sense of distributions, then 
		$\nabla u$ is locally BMO in $\Omega$, modulo the addition of a $\operatorname{curl}$-free field.
	\end{itemize}
	\end{corollary}

\paragraph{}\ The present article however leaves open the question whether regularity information can be transferred to the gradient also in the case of `borderline' spaces perfectly, i.e whether $u$ is locally Lipschitz or $C^{1}$. 
\begin{openquestion}\label{stein grad borderline} 
	For $u$ coclosed, if  
	\begin{align*}
	&\delta \left( a(x)\lvert d u \rvert^{p-2} d u \right) \in L_{loc}^{(n,1)}  \quad &&\Rightarrow ? \quad  u \in C_{loc}^{1}? \text{  } u \in C_{loc}^{0,1} ? 
	\end{align*}
\end{openquestion}

\paragraph*{} We conclude this introduction with a few words about the techniques and proofs. The main skeleton of the the linearization argument required to handle Dini continuous coefficients were first discovered in Kuusi-Mingione \cite{KuusiMingione_nonlinearStein}. The novelty here is use of gauge fixing procedures and employing the boundary value problems like \eqref{quasilinear maxwell tangential homogeneous general} and \eqref{quasilinear maxwell tangential homogeneous frozen general} instead of the usual Dirichlet problems to be able to use a Poincar\'{e}-Sobolev inequality, which in turns allows us to prove all the comparison estimates we need. Once this is achieved, the arguments in \cite{KuusiMingione_nonlinearStein} goes through virtually without change. So we shall primarily focus on proving the necessary ingredients and indicating \emph{only the  necessary changes} to their arguments and skip the details where it is just a matter of rewriting their arguments with obvious notational changes.    

\section{Preliminary material and notations}
\subsection{Notations}
We now fix the notations, for further details we refer to
\cite{CsatoDacKneuss} and \cite{hamburgerregularity}. Let $n \geq 2,$ $N \geq 1$ and $0 \leq k \leq n$ be integers. \begin{itemize}
	\item We write $\Lambda^{k}\left(  \mathbb{R}^{n}; \mathbb{R}^{N}\right)  $ (or simply
	$\Lambda^{k}$ if $N=1$) to denote the vector space of all alternating $k-$linear maps
	$f:\underbrace{\mathbb{R}^{n}\times\cdots\times\mathbb{R}^{n}}_{k-\text{times}%
	}\rightarrow\mathbb{R}^{N}.$ For $k=0,$ we set $\Lambda^{0}\left(  \mathbb{R}%
	^{n}; \mathbb{R}^{N}\right)  =\mathbb{R}^{N}.$ Note that $\Lambda^{k}\left(  \mathbb{R}%
	^{n}; \mathbb{R}^{N}\right)  =\{0\}$ for $k>n$ and, for $k\leq n,$ $\operatorname{dim}\left(
	\Lambda^{k}\left(  \mathbb{R}^{n}; \mathbb{R}^{N}\right)  \right)  ={\binom{{n}}{{k}}}\times N.$

	\item If $N=1,$ the symbols $\wedge,$ $\lrcorner\,,$ $\left\langle \ ;\ \right\rangle $ and,
	respectively, $\ast$ denote as usual the exterior product, the interior product, the
	scalar product and, respectively, the Hodge star operator. We extend these operations to vector-valued forms in the following way. For a scalar form $\eta \in \Lambda^{k}\left(  \mathbb{R}^{n}\right)  $ and a vector-valued form 
	$\xi \in \Lambda^{k}\left(  \mathbb{R}^{n}; \mathbb{R}^{N}\right) ,$ we define their exterior product and interior product componentwise, i.e. as
	$$ \eta \wedge \xi = \left(\eta \wedge \xi_{1}, \ldots, \eta \wedge \xi_{N} \right) \quad \text{ and }  \quad 
	\eta \lrcorner \xi = \left(\eta \lrcorner \xi_{1}, \ldots, \eta \lrcorner \xi_{N} \right). $$
	The scalar product extends to a scalar product between two vector-valued forms $\xi, \zeta \in \Lambda^{k}\left(  \mathbb{R}^{n}; \mathbb{R}^{N}\right) ,$ which is defined as 
	$$ \left\langle \xi ; \zeta \right\rangle = \sum_{i=1}^{N} \left\langle \xi_{i} ; \zeta_{i} \right\rangle. $$

\end{itemize}

\noindent Let $N \geqslant 1$, $0\leqslant k\leqslant n$ and let $\Omega\subset\mathbb{R}^n$ be open, bounded and smooth. 
\begin{itemize}
	\item An $\mathbb{R}^{N}$-valued differential $k$-form $u$ is a
	measurable function $u:\Omega\rightarrow\Lambda^{k} (\mathbb{R}^{n} ; \mathbb{R}^{N} ).$
	\item Two particular differential operators on differential forms will have a special significance for us. A differential
	$(k+1)$-form $\varphi\in L^{1}_{\rm loc}(\Omega;\Lambda^{k+1})$
	is called the exterior derivative of $\omega\in
	L^{1}_{\rm loc}\left(\Omega;\Lambda^{k}\right),$ denoted by $d\omega$,  if
	$$
	\int_{\Omega} \eta\wedge\varphi=(-1)^{n-k}\int_{\Omega} d\eta\wedge\omega,
	$$
	for all $\eta\in C^{\infty}_{0}\left(\Omega;\Lambda^{n-k-1}\right).$  The Hodge codifferential of $\omega\in L^{1}_{\rm loc}\left(\Omega;\Lambda^{k}\right)$ is
	a $(k-1)$-form, denoted $\delta\omega\in L^{1}_{\rm loc}\left(\Omega;\Lambda^{k-1}\right)$
	defined as
	$$
	\delta\omega:=(-1)^{nk+1}*d*\omega.
	$$ Of course, we set $d\omega \equiv 0$ when $k=n$ and $\delta \omega \equiv 0$ when $k=0.$ See \cite{CsatoDacKneuss} for the 
	properties and the integration by parts formula regarding these operators. We extend these definitions componentwise to the case of $\mathbb{R}^{N}$-valued forms. More precisely, 
	for any $\mathbb{R}^{N}$-valued $k$-form $\omega,$ the exterior derivative and the codifferential is defined as 
	$$ d\omega = \left( d\omega_{1}, \ldots, d\omega_{N} \right) \qquad \text{ and } \qquad  \delta\omega = \left( \delta\omega_{1}, \ldots, \delta\omega_{N} \right) . $$ The corresponding integration by parts 
	formulas extends componentwise as well.
	
	\item The usual Lebesgue, Sobolev and H\"{o}lder spaces and their local versions are defined componentwise in the usual way and are denoted by their usual symbols. The Morrey spaces, the Campanato spaces and the Lorentz spaces are defined in section \ref{function spaces}. 
	\item Let $ 1 \leq p \leq \infty$ and let $\nu$ be the outward unit normal to $\partial\Omega,$ identified with the $1$-form 
	$\displaystyle \nu = \sum_{i=1}^{n} \nu_{i} dx^{i} .$ For any $\mathbb{R}^{N}$-valued differential $k$-form $\omega$ on $\Omega,$ the $\mathbb{R}^{N}$-valued $(k+1)$-form 
	$\nu\wedge\omega$ on $\partial\Omega$ is called the tangential part of $\omega,$  interpreted in the sense of traces when $\omega$ is not continuous. The spaces $W_{T}^{1,p}\left(  \Omega;\Lambda^{k}\right)  $ and
	$W_{\delta, T}^{1,p}\left(  \Omega;\Lambda^{k}\right)  $ are defined as%
	\[
	W_{T}^{1,p}\left(  \Omega;\Lambda^{k}\right)  =\left\{  \omega\in
	W^{1,p}\left(  \Omega;\Lambda^{k}\right)  :\nu\wedge\omega=0\text{ on
	}\partial\Omega\right\},
	\]%
	\[W_{\delta, T}^{1,p}(\Omega; \Lambda^{k}) = \left\lbrace \omega \in W_{T}^{1,p}(\Omega; \Lambda^{k}) : \delta\omega = 0 \text{ in }
	\Omega \right\rbrace .
	\]
	The space $\mathcal{H}_{T}\left(  \Omega;\Lambda^{k}\right)$ is defined as
	\[
	\mathcal{H}_{T}\left(  \Omega;\Lambda^{k}\right)  =\left\{  \omega\in
	W_{T}^{1,2}\left(  \Omega;\Lambda^{k}\right)  :d\omega=0\text{ and }%
	\delta\omega=0\text{ in }\Omega\right\}
	\] We recall that when $\Omega$ is a contractible set, we have 
	$\mathcal{H}_{T}\left(  \Omega;\Lambda^{k}\right) = \left\lbrace 0 \right\rbrace.  $ In particular, this is true for any ball.
\end{itemize}
\subsection{Morrey, Campanato and Lorentz spaces}\label{function spaces}
Let $n \geq 2,$ $N \geq 1$ and $0 \leqslant k \leqslant n$ be integers. For $1\leqslant p <  \infty,$ $0 < \theta \leqslant \infty$ and $\lambda \geq 0,$  $\mathrm{L}^{p,\lambda}\left(\Omega;\left( \Lambda^{k}(\mathbb{R}^{n}); \mathbb{R}^{N} \right)\right) $ stands for the Morrey space of all $u \in L^{p}\left(\Omega;\left( \Lambda^{k}(\mathbb{R}^{n}); \mathbb{R}^{N} \right)\right)$ such that 
$$ \lVert u \rVert_{\mathrm{L}^{p,\lambda}\left(\Omega;\left( \Lambda^{k}(\mathbb{R}^{n}); \mathbb{R}^{N} \right)\right)}^{p} := \sup_{\substack{ x_{0} \in \overline{\Omega},\\ \rho >0 }} 
\rho^{-\lambda} \int_{B_{\rho}(x_{0}) \cap \Omega} \lvert u \rvert^{p} < \infty, $$ endowed with the norm 
$ \lVert u \rVert_{\mathrm{L}^{p,\lambda}}$ and $\mathcal{L}^{p,\lambda}\left(\Omega;\left( \Lambda^{k}(\mathbb{R}^{n}); \mathbb{R}^{N} \right)\right) $ denotes the Campanato space of all $u \in L^{p}\left(\Omega;\left( \Lambda^{k}(\mathbb{R}^{n}); \mathbb{R}^{N} \right)\right)$ such that 
$$ [u ]_{\mathcal{L}^{p,\lambda}\left(\Omega;\left( \Lambda^{k}(\mathbb{R}^{n}); \mathbb{R}^{N} \right)\right)}^{p} := \sup_{\substack{ x_{0} \in \overline{\Omega},\\  \rho >0 }} 
\rho^{-\lambda} \int_{B_{\rho}(x_{0}) \cap \Omega} \lvert u  - (u)_{ \rho , x_{0}}\rvert^{p} < \infty, $$ endowed with the norm 
$ \lVert u \rVert_{\mathcal{L}^{p,\lambda}} := \lVert  u \rVert_{L^{p}} +  
[u ]_{\mathcal{L}^{p,\lambda}}.$ Here 
$$\displaystyle (u)_{ \rho , x_{0}} = \frac{1}{\operatorname*{meas} \left( B_{\rho}(x_{0}) \cap \Omega \right)}\int_{B_{\rho}(x_{0}) \cap \Omega} u  = \fint_{B_{\rho}(x_{0}) \cap \Omega} u .$$ We remark that we would consistently use these notations for averages and averaged integrals throughout the rest.   
For standard facts about these spaces, particularly their identification with H\"{o}lder spaces and BMO space, see \cite{giaquinta-martinazzi-regularity}. A $\mathbb{R}^{N}$-valued differential $k$-form $u: \Omega \rightarrow \left( \Lambda^{k}(\mathbb{R}^{n}); \mathbb{R}^{N} \right)$ is said to belong to the Lorentz space $L^{(p,\theta)} \left(\Omega;\left( \Lambda^{k}(\mathbb{R}^{n}); \mathbb{R}^{N} \right)\right)$ if 
$$ \left\lVert u \right\rVert_{L^{(p,\theta)}\left(\Omega;\left( \Lambda^{k}(\mathbb{R}^{n}); \mathbb{R}^{N} \right)\right)}^{\theta} 
:= \int_{0}^{\infty} \left( t^{p} \left\lvert \left\lbrace x \in \Omega : \left\lvert u \right\rvert > t \right\rbrace\right\rvert \right)^{\frac{\theta}{p}} \frac{\mathrm{d}t}{t} < \infty $$
for $0 < \theta < \infty$ and if 
$$ \left\lVert u \right\rVert_{L^{(p,\infty)}\left(\Omega;\left( \Lambda^{k}(\mathbb{R}^{n}); \mathbb{R}^{N} \right)\right)}^{p} 
:= \sup_{t >0} \left( t^{p} \left\lvert \left\lbrace x \in \Omega : \left\lvert u \right\rvert > t \right\rbrace\right\rvert \right) < \infty .$$ For different properties of Lorentz spaces, see \cite{SteinWeiss_Fourieranalysis}.  
\subsection{The auxiliary mapping $V$}
As is standard in the literature, we shall use the following auxiliary mapping $V: \mathbb{R}^{\binom{n}{k+1}\times N} \rightarrow \mathbb{R}^{\binom{n}{k+1}\times N}$ 
defined by 
\begin{equation}\label{definition V}
V(z) : = \left\lvert z \right\rvert^{\frac{p-2}{2}} z , 
\end{equation}
which is a locally Lipschitz bijection from $\mathbb{R}^{\binom{n}{k+1}\times N}$ into itself. 
We summarize the relevant properties of the map in the following. 
\begin{lemma}\label{prop of V}
	For any $p >1$ and any $0 \leq k \leq n-1,$ there exists a constant $c_{V} = c_{V}(n,N,k,p) > 0$ such that 
	\begin{equation}\label{constant cv}
	\frac{\left\lvert z_{1} - z_{2}\right\rvert}{c_{V}} \leq  \frac{\left\lvert V(z_{1}) - V(z_{2})\right\rvert}{\left( \left\lvert z_{1} \right\rvert + 
		\left\lvert z_{2}\right\rvert \right)^{\frac{p-2}{2}}} \leq c_{V}\left\lvert z_{1} - z_{2}\right\rvert,
	\end{equation}
	for any $z_{1}, z_{2} \in \mathbb{R}^{\binom{n}{k+1}\times N},$ not both zero. This implies the classical monotonicity estimate 
	\begin{equation}\label{monotonicity}
	\left( \left\lvert z_{1} \right\rvert + \left\lvert z_{2}\right\rvert \right)^{p-2} \left\lvert z_{1} - z_{2}\right\rvert^{2} \leq c(n,N,k,p) 
	\left\langle  \left\lvert z_{1} \right\rvert^{p-2} z_{1} - \left\lvert z_{2} \right\rvert^{p-2} z_{2}, z_{1} - z_{2} \right\rangle ,  
	\end{equation}
	with a constant $c(n,N,k,p) > 0$ for all $p >1$ and all $z_{1}, z_{2} \in \mathbb{R}^{\binom{n}{k+1}\times N}.$
	Moreover, if $1 < p \leq 2,$ there exists a constant $c = c(n,N,k,p) > 0$ such that for any  $z_{1}, z_{2} \in \mathbb{R}^{\binom{n}{k+1}\times N},$
	\begin{equation}\label{v estimate p less 2}
	\left\lvert z_{1} - z_{2}\right\rvert \leq c \left\lvert V(z_{1}) - V(z_{2})\right\rvert^{\frac{2}{p}} + c \left\lvert V(z_{1}) - V(z_{2})\right\rvert 
	\left\lvert z_{2}\right\rvert^{\frac{2-p}{2}}.
	\end{equation}
\end{lemma}
The estimates \eqref{constant cv} and \eqref{monotonicity} are classical (cf. lemma 2.1, \cite{hamburgerregularity}). The estimate \eqref{v estimate p less 2} follows from this 
(cf. lemma 2, \cite{KuusiMingione_nonlinearStein}).

\subsection{Local and boundary regularity for the linear $d-\delta$ system}
Here we collect the local and up to the boundary linear estimates for the Hodge systems, sometimes also called $\operatorname*{div}$-$\operatorname*{curl}$ systems 
or $d-\delta$ systems, that we are going to use. Since we would mostly be using them for balls, we state them here for balls as well.  
We begin with the local estimates.  
\begin{theorem}[local estimates]\label{linearlocal estimates}
	Let $0 < r< R $ and $1 < r,s  < \infty, $ $0 < \lambda < n+2 $ be real numbers and $0 < \theta \leq \infty$. Suppose  $u \in W^{1, s} 
	\left(B_{R}; \Lambda^{k}\mathbb{R}^{n}\otimes \mathbb{R}^{N} \right)$ is a local solution to  
	\begin{align}\label{local linear system}
	\left\lbrace \begin{aligned}
	du &= f &&\text{ in } B_{R}, \\
	\delta u &=g &&\text{ in } B_{R} . 
	\end{aligned}\right. 
	\end{align}
	Then whenever $B(x,r) \subset B_{R}$ is a ball (not necessarily concentric to $B_{R}$), we have the following. 
	\begin{itemize}
		\item[(i)] If $f \in L^{(r,\theta)}_{loc} \left(B_{R}; \Lambda^{k+1}(\mathbb{R}^{n}; \mathbb{R}^{N}) \right)$ and $g \in L^{(r,\theta)}_{loc} \left(B_{R}; \Lambda^{k-1}(\mathbb{R}^{n}; \mathbb{R}^{N}) \right)$ then 
		$u \in W_{loc}^{1, (r,\theta )} \left(B_{R}; \Lambda^{k+1}(\mathbb{R}^{n}; \mathbb{R}^{N}) \right)$ and there exists a constant $c > 0,$ depending only on $n,N,k,\theta $ such that we have the estimate
		\begin{align}\label{local estimate Lp}
		\left\lVert \nabla u \right\rVert_{L^{r,\theta}(B(x,r/2))} \leq c \left(  \left\lVert f \right\rVert_{L^{r,\theta}(B(x,r))}  + \left\lVert g \right\rVert_{L^{r,\theta}(B(x,r))} + \left\lVert \nabla u \right\rVert_{L^{s}(B(x,r))}\right). 
		\end{align}
		
		\item[(ii)] If $f \in \mathcal{L}^{2,\lambda }_{loc} \left(B_{R}; \Lambda^{k+1}(\mathbb{R}^{n}; \mathbb{R}^{N}) \right)$ and $g \in \mathcal{L}^{2,\lambda }_{loc} \left(B_{R}; \Lambda^{k-1}(\mathbb{R}^{n}; \mathbb{R}^{N}) \right),$  then 
		$\nabla u \in \mathcal{L}^{2,\lambda }_{loc}\left( B_{R}; \Lambda^{k}\mathbb{R}^{n}\otimes \mathbb{R}^{N} \right)$ and there exists a constant $c > 0,$ depending only on $n,N,k,\lambda$ such that we have the estimate
		\begin{align}\label{local estimate campanato}
		\left\lVert \nabla u \right\rVert_{\mathcal{L}^{2,\lambda }(B(x,r/2))} \leq  c\left( \left\lVert f \right\rVert_{\mathcal{L}^{2,\lambda }(B(x,r))} + \left\lVert g \right\rVert_{\mathcal{L}^{2,\lambda }(B(x,r))} + \left\lVert \nabla u \right\rVert_{L^{s}(B(x,r))} \right) . 
		\end{align}
	\end{itemize} 
\end{theorem}
The result is standard local estimates for constant coefficient elliptic system (cf. for example, theorem 5.14 of 
\cite{giaquinta-martinazzi-regularity}), the extension to Lorentz spaces follows in the usual way via interpolation.

\paragraph*{} Now we turn to boundary estimates, which are often also called Gaffney or Gaffney-Friedrichs inequality. 
Both estimates follows from the $L^{p}$ and Schauder estimates for the constant coefficient linear elliptic system \eqref{local linear system}  and goes back to 
Morrey \cite{MorreyHarmonic2} (see also \cite{bolik}, \cite{Sil_linearregularity} for linear estimates for more general Hodge type systems). The $L^{p}$ estimates extend to the scale of Lorentz spaces by interpolation and this is the form in which we state the results. 

\begin{theorem}[boundary estimates]\label{linear boundary estimates}
	Let $R > 0$ and $1 < r,s  < \infty, $ $0 < \lambda < n+2 $ be real numbers and $0 < \theta \leq \infty$. Suppose  $u \in W^{1, s} 
	\left(B_{R}; \Lambda^{k}\mathbb{R}^{n}\otimes \mathbb{R}^{N} \right)$ is a solution to 
	\begin{align}\label{boundary linear system}
	\left\lbrace \begin{aligned}
	du &= f &&\text{ in } B_{R}, \\
	\delta u &=g &&\text{ in } B_{R} \\
	\nu \wedge u &= 0 &&\text{ on } \partial B_{R}. 
	\end{aligned}\right. 
	\end{align}
	\begin{itemize}
		\item[(i)] If $f \in L^{(r,\theta)} \left(B_{R}; \Lambda^{k+1}(\mathbb{R}^{n}; \mathbb{R}^{N}) \right)$ and $g \in L^{(r,\theta)} \left(B_{R}; \Lambda^{k-1}(\mathbb{R}^{n}; \mathbb{R}^{N}) \right)$ then 
		$u \in W^{1, (r,\theta )} \left(B_{R}; \Lambda^{k+1}(\mathbb{R}^{n}; \mathbb{R}^{N}) \right)$ and there exists a constant $c > 0,$ depending only on $n,N,k,\theta $ such that we have the estimate
		\begin{align}\label{linear boundary Lp}
		\left\lVert \nabla u \right\rVert_{L^{r,\theta}(B_{R})} \leq c \left(  \left\lVert f \right\rVert_{L^{r,\theta}(B_{R})}  + \left\lVert g \right\rVert_{L^{r,\theta}(B_{R})}\right). 
		\end{align}
		
		\item[(ii)] If $f \in \mathcal{L}^{2,\lambda } \left(B_{R}; \Lambda^{k+1}(\mathbb{R}^{n}; \mathbb{R}^{N}) \right)$ and $g \in \mathcal{L}^{2,\lambda } \left(B_{R}; \Lambda^{k-1}(\mathbb{R}^{n}; \mathbb{R}^{N}) \right),$  then 
		$\nabla u \in \mathcal{L}^{2,\lambda }\left( B_{R}; \Lambda^{k}\mathbb{R}^{n}\otimes \mathbb{R}^{N} \right)$ and there exists a constant $c > 0,$ depending only on $n,N,k,\lambda$ such that we have the estimate
		\begin{align}\label{linear boundary holder}
		\left\lVert \nabla u \right\rVert_{\mathcal{L}^{2,\lambda }(B_{R})} \leq  c\left( \left\lVert f \right\rVert_{\mathcal{L}^{2,\lambda }(B_{R})} + \left\lVert g \right\rVert_{\mathcal{L}^{2,\lambda }(B_{R})} \right) . 
		\end{align}
	\end{itemize}
\end{theorem}
\begin{remark} Note that there is no term containing $u$ on the right hand side of the estimates above as the domain is ball, which being contractible has 
	only trivial DeRham cohomology. Thus, the system \eqref{boundary linear system} has uniqueness and the usual term containing $u$ on the right hand side can be dropped by the standard 
	contradiction-compactness argument. \end{remark}

The above estimates combined with the Sobolev inequality and a contradiction argument yields the following Poincar\'{e}-Sobolev inequality, which is crucial for our  purposes.  
\begin{proposition}[Poincar\'{e}-Sobolev inequality]\label{poincaresobolev}
	Let $N \geq 1,$ $n \geq 2$ and $1 \leq k \leq n-1$ be integers. Let $R > 0$ and $1 < s < \infty$ be real numbers. Then for any $u \in W_{\delta,T}^{1,s}\left( B_{R} ; \Lambda^{k}\mathbb{R}^{n}\otimes \mathbb{R}^{N}\right),$ there exists a constant $c >0,$ depending only on $k,n,N,s$ such that 
	\begin{equation}\label{poincaresobolevineq}
	\left( \fint_{B_{R}} \lvert u \rvert^{s^*} \right)^{\frac{1}{s^*}} \leq c R \left( \fint_{B_{R}} \lvert du \rvert^{s} \right)^{\frac{1}{s}}. 
	\end{equation}
\end{proposition}
\begin{proof}
	Clearly, it is enough to prove the result for $R=1.$ Now since $u \in W_{\delta,T}^{1,s}\left( B_{1} ; \Lambda^{k}\mathbb{R}^{n}\otimes \mathbb{R}^{N}\right),$ \eqref{linear boundary Lp} with $g=0$ and $r=\theta = s$ implies the estimate 
	\begin{equation}
	\lVert \nabla u \rVert_{L^{s}} \leq c \lVert du \rVert_{L^{s}}.
	\end{equation}
	In view of the Sobolev inequality, this implies the desired result as soon as we prove the Poincar\'{e} inequality
	$$ \lVert u \rVert_{L^{s}} \leq c \lVert \nabla u \rVert_{L^{s}}.$$ 
	But if this is not true, then there exists a sequence $u_{\mu}$ such that $\left\lVert u_{\mu} \right\rVert_{L^{s}} = 1$ and 
	$\left\lVert \nabla u_{\mu} \right\rVert_{L^{s}} \leq \frac{1}{\mu} $for all $\mu \geq 1.$ Thus $\left\lVert u_{\mu} \right\rVert_{W^{1, s}}$ is uniformly bounded and consequently $u_{\mu} \stackrel{W^{1,s}}{\rightharpoonup} u$ for some $u \in W^{1,s}.$ But then by compact Sobolev embedding, 
	$\left\lVert u \right\rVert_{L^{s}} = 1$ and $ \left\lVert \nabla u \right\rVert_{L^{s}} \leq \liminf \left\lVert \nabla u_{\mu} \right\rVert_{L^{s}} = 0.$ But this implies $u$ is a constant form. But no non-zero constant form can satisfy the boundary condition $\nu\wedge u = 0$ on $\partial B_{1}= S^{n-1}.$ Thus $u\equiv 0$ and this contradicts the fact that $\left\lVert u \right\rVert_{L^{s}} = 1.$ 
\end{proof}

\subsection{On existence and weak formulations}\label{existencediscussion} Throughout the rest of the article, we shall often start with a local weak solution $u \in W_{loc}^{1,p}\left( \Omega ; \Lambda^{k}\mathbb{R}^{n}\otimes \mathbb{R}^{N}\right)$ of quasilinear systems of the type 
\begin{align}\label{existence system}
\delta \left( a(x) \left\lvert du \right\rvert^{p-2} du \right) &= f  &&\text{ in } \Omega, \tag{$\mathcal{P}$}
\end{align}
with or without the additional condition $\delta u = 0 $ in $\Omega.$ So whether such a solution \emph{exists} is the first order of business. The existence is actually not as straight forward as one might think, since trying to minimize the corresponding energy functional over $W^{1,p},$ with, say, homogeneous Dirichlet boundary values, one immediately realizes that the functional control only the $L^{p}$ norm of $du$ and thus is not coercive on $W^{1,p}.$ However,  one can still show ( see \cite{BandDacSil} for $N=1,$ \cite{Sil_semicontinuity} for the general case ) the existence of a minimizer for the following two minimization problems
\begin{gather}
m = \inf\left\lbrace \int_{\Omega} \left( a(x)\left\lvert du \right\rvert^{p} - \langle F ; du \rangle \right): u \in u_{0} + W^{1,p}_{0}\left( \Omega ; \Lambda^{k}\mathbb{R}^{n}\otimes \mathbb{R}^{N}\right) \right\rbrace \label{w01p minimization}\\\text{ and } \notag \\ m = \inf\left\lbrace \int_{\Omega} \left( a(x)\left\lvert du \right\rvert^{p}- \langle F ; du \rangle \right): u \in u_{0} + W^{1,p}_{\delta, T}\left( \Omega ; \Lambda^{k}\mathbb{R}^{n}\otimes \mathbb{R}^{N}\right) \right\rbrace, \label{wdeltaT1p minimization} 
\end{gather}
as long as $ F \in L^{p'}\left( \Omega ; \Lambda^{k+1}(\mathbb{R}^{n}; \mathbb{R}^{N})\right).$ But since $\delta f = 0$ ( in the sense of distributions ) is clearly a necessary condition for solving \eqref{existence system}, we can take $F = d\theta \in W^{1,d^{*}} \hookrightarrow L^{p'},$  where $\theta \in W^{2,d}$ is the unique solution of 
\begin{align*}
\left\lbrace \begin{aligned}
\delta d \theta &= f &&\text{ in } \Omega, \\
\delta \theta &= 0  &&\text{ in } \Omega, \\
\nu \wedge \theta &= 0  &&\text{ on } \partial\Omega, 
\end{aligned}\right.
\end{align*}
which exists ( see e.g. \cite{Sil_linearregularity} ) as long as $f \in L^{d}$ is coclosed, where $d$ is the exponent given by 
\begin{align}\label{ddef}
d := \left\lbrace \begin{aligned}
&\frac{np}{np -n + p} &&\text{ if } 1< p < n, \\
& 1+ \varepsilon &&\text{ if } p \geq n, \text{ for any } \varepsilon >0 .  \\
\end{aligned} \right. 
\end{align} 
Then we can write, since $\delta F = f$ in $\Omega,$ 
\begin{align*}
\int_{\Omega} \left( a(x)\left\lvert du \right\rvert^{p} + \langle f ; u \rangle \right) &= \int_{\Omega} \left( a(x)\left\lvert du \right\rvert^{p} - \langle F ; du \rangle \right) + \int_{\partial\Omega}   \langle F ; \nu \wedge u_{0} \rangle .
\end{align*}
Since $f,u_{0}$ are given data, the last integral is a constant irrelevant for the minimization. Note that the minimizer to \eqref{wdeltaT1p minimization} satisfies $\delta u = 0$ and is unique by \eqref{poincaresobolev}. This is the one we shall be using the most. For this minimization problem, clearly the space of test function is $W^{1,p}_{\delta, T}$ and the weak formulation is 
\begin{align*}
\int_{\Omega} \left\langle a(x)\left\lvert du \right\rvert^{p-2}d u - F; d\phi \right\rangle = 0 \quad \text{ for all } \phi \in W^{1,p}_{\delta, T}\left( \Omega ; \Lambda^{k}\mathbb{R}^{n}\otimes \mathbb{R}^{N}\right),
\end{align*}  
which, by our definition of $F$ is easily seen to be equivalent to, 
\begin{align*}
\int_{\Omega} \left\langle a(x)\left\lvert du \right\rvert^{p-2}d u ; d\phi \right\rangle = - \int_{\Omega} \left\langle f ; \phi \right\rangle \quad \text{ for all } \phi \in W^{1,p}_{\delta, T}\left( \Omega ; \Lambda^{k}\mathbb{R}^{n}\otimes \mathbb{R}^{N}\right).
\end{align*}
Note also that the integral on the right makes sense by \eqref{ddef}. We summarize the preceding discussion in the following 
\begin{proposition}\label{minimizerexistenceprop}
	Let $d$ be the exponent in \eqref{ddef} and 
	$a:\Omega \rightarrow [\gamma, L]$, where $ 0 < \gamma < L < \infty,$ is a measurable map. Then for any ball $B_{R},$ any $1 < p < \infty ,$ any $u_{0} \in W^{1,p}\left( \Omega ; \Lambda^{k}\mathbb{R}^{n}\otimes \mathbb{R}^{N}\right)$ and for any $f \in L^{d}\left( \Omega ; \Lambda^{k}\mathbb{R}^{n}\otimes \mathbb{R}^{N}\right), $ the quasilinear boundary value problem 
	\begin{align}
	\left\lbrace \begin{aligned}
	\delta \left( a(x) \left\lvert d u \right\rvert^{p-2} du \right) &= f &&\text{ in } B_{R}, \\
	\delta u &= \delta u _{0}  &&\text{ in } B_{R}, \\
	\nu \wedge u &= \nu \wedge u_{0}  &&\text{ on } \partial B_{R}, 
	\end{aligned}\right.
	\end{align}
	admits a unique solution $u \in u_{0} + W^{1,p}_{\delta, T}\left( B_{R} ; \Lambda^{k}\mathbb{R}^{n}\otimes \mathbb{R}^{N}\right). $ Moreover, the solution is the unique minimizer to the minimization problem 
$$m = \inf\left\lbrace \int_{B_{R}} \left( a(x)\left\lvert du \right\rvert^{p} +  \langle f ; u \rangle \right): u \in u_{0} + W^{1,p}_{\delta, T}\left( B_{R} ; \Lambda^{k}\mathbb{R}^{n}\otimes \mathbb{R}^{N}\right) \right\rbrace . $$ \end{proposition}\smallskip 

\subsubsection{Enlarging the space of test functions} For our techniques, it is crucial that we work with the weak formulation on the space $W^{1,p}_{\delta, T}\left( B_{R} ; \Lambda^{k}\mathbb{R}^{n}\otimes \mathbb{R}^{N}\right)$ instead of the usual space $W^{1,p}_{0}\left( B_{R} ; \Lambda^{k}\mathbb{R}^{n}\otimes \mathbb{R}^{N}\right).$ So we need to show that indeed such a weak formulation is valid. 
\begin{proposition}\label{weakformulationprop}
	Let $u \in W_{loc}^{1,p}\left( \Omega ; \Lambda^{k}\mathbb{R}^{n}\otimes \mathbb{R}^{N}\right)$ be a local weak solution of \eqref{existence system}, i.e. for any ball $B_{R}\subset \subset \Omega,$ $u$ satisfies the weak formulation 
	\begin{align}\label{weak formulation in W1p}
	\int_{B_{R}} \left\langle a(x) \left\lvert du \right\rvert^{p-2} du; d\psi \right\rangle = \int_{B_{R}} \left\langle f; \psi \right\rangle \qquad \text{ for all } \psi \in W_{0}^{1,p}\left( B_{R} ; \Lambda^{k}\mathbb{R}^{n}\otimes \mathbb{R}^{N}\right).
	\end{align} Then $u$ also satisfies 
	\begin{align}\label{weak formulation}
	\int_{B_{R}} \left\langle a(x) \left\lvert du \right\rvert^{p-2} du; d\phi \right\rangle = \int_{B_{R}} \left\langle f; \phi \right\rangle \qquad \text{ for all } \phi \in W_{\delta, T}^{1,p}\left( B_{R} ; \Lambda^{k}\mathbb{R}^{n}\otimes \mathbb{R}^{N}\right).
	\end{align} 
\end{proposition}
\begin{proof}
	Given any $\phi \in W_{\delta, T}^{1,p}\left( B_{R} ; \Lambda^{k}\mathbb{R}^{n}\otimes \mathbb{R}^{N}\right),$ the trick is to write $\phi = \psi + d\theta ,$ with $\psi \in W_{0}^{1,p}\left( B_{R} ; \Lambda^{k}\mathbb{R}^{n}\otimes \mathbb{R}^{N}\right).$ To this end, using the fact that $\nu\wedge \phi = 0 $ on $\partial B_{R},$ we first solve ( cf. Theorem 8.16 of \cite{CsatoDacKneuss} for $p \geq 2$ and Theorem 2.47 of \cite{silthesis}  for the general case ),  
	\begin{align*}
	\left\lbrace \begin{aligned}
	d\psi &= d\phi &&\text{ in } B_{R}\\
	\psi &= 0 &&\text{ on } \partial B_{R}. 
	\end{aligned}\right. 
	\end{align*}
	Next, we solve (cf. Theorem 9 in \cite{Sil_linearregularity} ) 
	\begin{align*}
	\left\lbrace \begin{aligned}
	\delta d\theta &=  - \delta \psi &&\text{ in } B_{R}\\
	\delta \theta &= 0 &&\text{ in } B_{R}\\ 
	\nu\wedge\theta &= 0 &&\text{ on } \partial B_{R}. 
	\end{aligned}\right. 
	\end{align*}
	Now it is easy to check that $\phi - \psi - d\theta \in \mathcal{H}_{T}\left( B_{R} ; \Lambda^{k}\mathbb{R}^{n}\otimes \mathbb{R}^{N}\right)$ and hence is identically zero in $B_{R}.$ 
\end{proof}

\subsection{Regularity for the homogeneous constant coefficient system}
We begin with the classical estimates for a constant coefficient homogeneous system, which essentially goes back to Uhlenbeck \cite{uhlenbecknonlinearelliptic} ( see also  \cite{hamburgerregularity} ). Let 
$v \in W_{loc}^{1,p}\left( \Omega ; \Lambda^{k}\mathbb{R}^{n}\otimes \mathbb{R}^{N}\right)$ to a local solution to 
\begin{align}\label{p laplacian homogeneous frozen}
\delta ( a(x_{0}) \lvert dv \rvert^{p-2} dv)   = 0   \qquad \text{ in } \Omega.
\end{align}
The following two results are essentially proved in theorem 3 and lemma 3 of \cite{KuusiMingione_nonlinearStein}, respectively.
\begin{theorem}\label{theoremhomogeneousfrozen}
	Let $v$ be as in \eqref{p laplacian homogeneous frozen}, then $dv$ is locally 
	H\"{o}lder continuous (with an exponent $\beta_{1}$ given below) on 
	$\Omega$. Moreover, 
	\begin{itemize}
		\item[(i)] There exist constants $\bar{c_{2}} \equiv \bar{c_{2}}\left( n, N, k, p , \gamma, L \right)$ and $\beta_{2} \equiv \beta_{2} \left( n, N, k, p , \gamma, L \right) \in (0,1)$
		such that the estimate 
		\begin{multline}\label{Vdvscaling}
		\left( \fint_{B (x_{0}, \rho)} \left\lvert  V(d v) - \left( V(d v) \right)_{B(x_{0},\rho)} \right\rvert^{2} \right)^{\frac{1}{2}}  \\ 
		\leq \bar{c_{2}}\left( \frac{\rho}{R}\right)^{\beta_{2}} \left( \fint_{B(x_{0},R)} \left\lvert  V(d v) - \left( V(d v) \right)_{B(x_{0},R)} \right\rvert^{2} \right)^{\frac{1}{2}} 
		\end{multline}
		holds for whenever $0 < \rho < R$ and $B(x,R) \subset \Omega. $
		\item[(ii)] There exists a constant $c_{1} \geq 1$ 
		depending only on $n,N, k, p, \gamma, L $ such that the estimate 
		\begin{equation}\label{supestimatefrozenhomogeneous}
		\sup\limits_{B(x,R/2)} \left\lvert dv \right\rvert  \leq c_{1} \fint_{B(x,R)} \left\lvert d v \right\rvert
		\end{equation}
		holds whenever $B(x,R) \subset \Omega. $
		\item[(iii)] For every $A \geq 1$ there exist constants
		$$ c_{2} \equiv c_{2}\left( n, N, k, p, \gamma, L, A \right) \equiv \widetilde{c}_{2}\left( n, N, k, p , \gamma, L \right)A $$
		and $\beta_{1} \equiv \beta_{1} \left( n, N, k, p , \gamma, L \right) \in (0,1)$ such that 
		\begin{gather}\label{oscillationcontrolfrozenhomogeneous}
		\sup\limits_{B(x,R/2)} \left\lvert d v \right\rvert \leq A\lambda  \Longrightarrow \quad 
		\operatorname*{osc}\limits_{B(x, \tau R)} \left( d v \right) 
		\leq c_{2}\tau^{\beta_{1}}\lambda  \end{gather}
		for every $\tau \in ( 0, \frac{1}{2}).$ 
	\end{itemize}
	
\end{theorem}
\begin{lemma}\label{gradvscalinggeneral}
	Let $v$ be as in \eqref{p laplacian homogeneous frozen}. Then for every choice of $0 < \bar{\varepsilon} < 1$ and $A \geq 1$ there exists a constant 
	$\sigma_{2} \in (0, 1/2)$ depending only on $n,k,p,\bar{\varepsilon} $ and $A,$ such that if $\sigma \in (0, \sigma_{2}]$ and 
	\begin{equation}\label{bothsideboundfrozengeneral}
	\frac{\lambda}{A} \leq  \sup\limits_{ B(x, \sigma R)} \left\lvert d v \right\rvert \leq  \sup\limits_{B (x, R /2 ) } \left\lvert d v \right\rvert \leq A\lambda,
	\end{equation}
	then 
	\begin{equation}\label{gradvscalingequationgeneral}
	\left( \fint_{B(x, \sigma R)} \left\lvert  d v - \left( d v \right)_{\sigma B_{j}} \right\rvert^{t} \right)^{\frac{1}{t}} 
	\leq \bar{\varepsilon}\left( \fint_{B(x,R)} \left\lvert d v - \left( d v \right)_{B(x,R)}\right\rvert^{t} \right)^{\frac{1}{t}}
	\end{equation}
	whenever $t \in [1,2]$ and $B(x,R) \subset \Omega. $
\end{lemma}

\section{Nonlinear Stein theorem for forms}
\subsection{Homogeneous system with Dini coefficients}
In this section, we prove continuity of the exterior derivative for the homogeneous system with Dini continuous coefficients that we shall use in the proof of the general case. We would be using these intermediate results only for the 
case $p >2,$ so we focus only on that case for now.\smallskip 

\noindent Let $p>2$ and $w \in W_{loc}^{1,p}\left( \Omega ; \Lambda^{k}\mathbb{R}^{n}\otimes \mathbb{R}^{N}\right)$ to be a local solution to 
\begin{align}\label{quasilinear maxwell tangential homogeneous dini}
\delta ( a(x) \lvert dw \rvert^{p-2} dw  )  = 0   \qquad \text{ and }
\delta w = 0 \qquad \text{ in } \Omega. 
\end{align}
\begin{theorem}\label{dwcontinuityhomogeneousDini}
	Let $w \in W_{loc}^{1,p}\left( \Omega ; \Lambda^{k}\mathbb{R}^{n}\otimes \mathbb{R}^{N}\right)$ be as in \eqref{quasilinear maxwell tangential homogeneous dini} with $p >2,$ 
	then $dw$ is continuous in $\Omega$. 
	Moreover, 
	\begin{itemize}
		\item[(i)] There exists a constant 
		$c_{3} = c_{3}(n,N, k,p, \gamma, L, \omega( \cdot )) \geq 1$ and a positive radius $R_{1}= R_{1}(n,N, k,p, \gamma, L, \omega( \cdot )) >0$ such that if $R \leq R_{1}$, then the 
		estimate 
		$$ \sup\limits_{B(x_{0}, R/2)} \left\lvert d w \right\rvert \leq c_{3}  \fint_{B(x_{0}, R)} \left\lvert dw \right\rvert, $$
		holds whenever $B(x_{0},R) \subset \Omega.$ If $a(\cdot )$ is a constant function, the estimate holds without any restriction on 
		$R.$
		\item[(ii)] Assume that the inequality 
		$$ \sup\limits_{B(x_{0}, R/2)} \left\lvert d w \right\rvert \leq A\lambda $$
		hold for some $A \geq 1$ and $\lambda > 0.$ Then for any $\delta \in (0,1)$ there exists a positive constant 
		$\sigma_{3} \equiv \sigma_{3} (n,N, k,p, \gamma, L, \omega( \cdot ), A, \delta ) \in (0, \frac{1}{4})$ such that for every $\sigma \leq \sigma_{3},$ we have,  
		$$ \operatorname*{osc}\limits_{B(x_{0}, \sigma R)} d w \leq \delta \lambda .$$
	\end{itemize}
\end{theorem}
\begin{remark} The corresponding theorem for $p$-Laplacian system that has been used in \cite{KuusiMingione_nonlinearStein} is deduced from its parabolic analogue proved in \cite{KuusiMingione_perutbationparabolic}. Here we give a direct proof of this result.  
	\end{remark}
	
\subsubsection{General setting for the proofs} Let $B(x, 2R) \subset \Omega$ be a fixed ball and for $i \geq 0,$ we set 

\begin{equation}\label{shrinkingballs1}
B_{i} \equiv B(x, R_{i}), \qquad R_{i}:= \sigma^{i} R, \qquad \sigma \in (0, \frac{1}{2}) .
\end{equation}
Now we define the maps $v_{i} \in w + W_{\delta,T}^{1,p}\left( B_{i} ; \Lambda^{k}\mathbb{R}^{n}\otimes \mathbb{R}^{N}\right)$ to be the unique solution of 
\begin{align}\label{quasilinear maxwell tangential homogeneous frozen i}
\left\lbrace \begin{aligned}
\delta ( a(x_{0}) \lvert dv_{i} \rvert^{p-2} dv_{i}) )  &= 0   &&\text{ in } B_{i},\\
\delta v_{i} &= 0 &&\text{ in } B_{i}, \\
\nu\wedge v_{i} &= \nu\wedge w &&\text{  on } \partial B_{i}.                \end{aligned} 
\right. 
\end{align}
Also, we define the quantities, for $i \geq 0$ and $r \geq 1,$ 
$$ m_{i}(G) := \left\lvert \left( G \right)_{B_{i}}\right\rvert \qquad \text{ and } \qquad 
E_{r}(G, B_{i}):=  \left( \fint_{B_{i}} \lvert G -  \left( G \right)_{B_{i}} \rvert^{r}\right)^{\frac{1}{r}}.$$
\begin{lemma}\label{secondcomparison1}
	Let $w,v_{i}$ be as before and  $i \geq 0.$ Then there exists a constant $c_{4} \equiv c_{4}\left( n, N, k, p, \gamma, L \right)$ such that we have the inequality
	\begin{align}\label{l1omega}
	\fint_{B_{i}} \lvert V(dv_{i}) - V(dw) \rvert^{2}  &\leq c_{4}\left[ \omega\left(R_{i}\right) \right]^{2}  
	\fint_{B_{i}} \lvert dw \rvert^{p}.
	\end{align}
\end{lemma}
\begin{proof}
	Weak formulation of \eqref{quasilinear maxwell tangential homogeneous dini} and \eqref{quasilinear maxwell tangential homogeneous frozen i} ( see Proposition \ref{weakformulationprop} ) gives 
	\begin{align*}
	\fint_{B_{i}} \left\langle a(x_{0})\left\lvert dv_{i}\right\rvert^{p-2}dv_{i} - a(x)\left\lvert dw\right\rvert^{p-2}dw; d\phi \right\rangle = 0, 
	\end{align*}
	for every $\phi \in W^{1,p}_{\delta, T}\left(B_{i};\Lambda^{k}\mathbb{R}^{n}\otimes \mathbb{R}^{N} \right).$
	Rewriting, we obtain,   
	\begin{multline*}
	\fint_{B_{i}} a(x_{0}) \left\langle \left\lvert dv_{i}\right\rvert^{p-2}dv_{i} - \left\lvert dw\right\rvert^{p-2}dw; d\phi \right\rangle \\= 
	\fint_{B_{i}} \left( a(x) - a(x_{0}) \right) \left\langle \left\lvert dw\right\rvert^{p-2}dw; d\phi \right\rangle 
	\end{multline*}
	for every $\phi \in W^{1,p}_{\delta, T}\left(B_{i};\Lambda^{k}\mathbb{R}^{n}\otimes \mathbb{R}^{N} \right).$
	Now we chose $\phi = v_{i}-w$ and by \eqref{constant cv} and using the fact that $a(x_{0}) \geq \gamma > 0,$ we have, by Young's inequality with $\varepsilon >0,$ 
	\begin{align*}
	\fint_{B_{i}}\lvert V(dv_{i}) - V(dw) \rvert^{2}  &\leq c \left\lvert  \fint_{B_{i}} \left( a(x) - a(x_{0}) \right) \left\langle \left\lvert dw\right\rvert^{p-2}dw; dv_{i}-dw \right\rangle\right\rvert \\
	&\leq c \omega\left(R_{i}\right) \fint_{B_{i}}\left\lvert dw\right\rvert^{p-1} \left\lvert dv_{i}-dw\right\rvert \\
	&\leq c  \omega\left(R_{i}\right) \fint_{B_{i}} \left( \left\lvert dw\right\rvert + \left\lvert dv_{i} \right\rvert  \right)^{p-1} \left\lvert dv_{i}-dw\right\rvert \\
	&= c  \omega\left(R_{i}\right) \fint_{B_{i}} \left( \left\lvert dw\right\rvert + \left\lvert dv_{i} \right\rvert  \right)^{\frac{p}{2}}
	\left( \left\lvert dw\right\rvert + \left\lvert dv_{i} \right\rvert  \right)^{\frac{p-2}{2}}\left\lvert dv_{i}-dw\right\rvert \\
	&\leq   c  \omega\left(R_{i}\right) \fint_{B_{i}} \left( \left\lvert dw\right\rvert + \left\lvert dv_{i} \right\rvert  \right)^{\frac{p}{2}}\lvert V(dv_{i}) - V(dw) \rvert \\
	&\leq \varepsilon \fint_{B_{i}}\lvert V(dv_{i}) - V(dw) \rvert^{2} 
	+ c\left[  \omega\left(R_{i}\right) \right]^{2} \fint_{B_{i}} \left( \left\lvert dw\right\rvert 
	+ \left\lvert dv_{i} \right\rvert  \right)^{p}.
	\end{align*}
	Since the functional $v \mapsto a(x_{0}) \int_{B_{i}} \left\lvert dv \right\rvert^{p}$ 
	on $ w +W^{1,p}_{\delta, T}\left(B_{i};\Lambda^{k}\mathbb{R}^{n}\otimes \mathbb{R}^{N} \right)$ is minimized by $v_{i},$ ( see Proposition \ref{minimizerexistenceprop} ) we have 
	\begin{equation}\label{v minimizes}
	\int_{B_{i}} \left\lvert dv_{i} \right\rvert^{p} \leq c \int_{B_{i}} \left\lvert dw \right\rvert^{p}
	\end{equation}
	Thus, we have, 
	\begin{align*}
	\fint_{B_{i}}\lvert V(dv_{i}) - V(dw) \rvert^{2} \leq \varepsilon \fint_{B_{i}}\lvert V(dv_{i}) - V(dw) \rvert^{2} 
	+ \frac{c \left[  \omega\left(R_{i}\right) \right]^{2}}{\varepsilon} \fint_{B_{i}} \left\lvert dw\right\rvert^{p}. 
	\end{align*}
	Choosing $\varepsilon > 0$ suitably small, we have \eqref{l1omega}. 
\end{proof}

\subsubsection{Pointwise bound}
\begin{theorem}\label{dwsupboundhomogeneousDini}
	Let $w \in W_{loc}^{1,p}\left( \Omega ; \Lambda^{k}\mathbb{R}^{n}\otimes \mathbb{R}^{N}\right)$ be as in \eqref{quasilinear maxwell tangential homogeneous dini} with $p >2.$ Then there exists a constant 
	$c = c(n,N, k,p, \gamma, L, \omega( \cdot )) \geq 1$ and a positive radius $R_{1}= R_{1}(n,N, k,p, \gamma, L, \omega( \cdot )) >0$ such that the pointwise estimate 
	$$ \left\lvert d w (x) \right\rvert \leq c \left( \fint_{B(x, R)} \left\lvert dw \right\rvert^{p}\right)^{\frac{1}{p}}, $$
	holds whenever $B(x,2R) \subset \Omega$, $2R \leq R_{1}$ and $x$ is a Lebesgue point of $d w.$ If $a(\cdot )$ is a constant function, the estimate holds without any restriction on 
	$R.$
\end{theorem}

\begin{proof}  \emph{ \textbf{Step 1: Choice of constants}} We pick an arbitrary point $x \in \Omega$ and a arbitrary positive radius $R_{1} > 0$ such that $B(x, R_{1}) \subset \Omega.$ We pick $0< R < R_{1}/2$ and for now set 
	$B(x,R)$ as our starting ball and consider the chain of shrinking balls as explained in \eqref{shrinkingballs1} for some parameter $\sigma \in (0,\frac{1}{4}).$ 
	We shall soon make specific choices of both the parameters $R_{1}$ and $\sigma.$
	We define the constant $\lambda$ as 
	$$ \lambda^{\frac{p}{2}}:= H_{1}\left( \fint_{B(x,R)} \left\lvert V(dw)\right\rvert^{2}\right)^{\frac{1}{2}} ,$$
	where $H_{1}$ will be chosen soon. 
	Clearly, we can assume $\lambda > 0.$ In view of theorem \ref{theoremhomogeneousfrozen}, we choose $\sigma \in (0,\frac{1}{4})$ small enough such that 
	\begin{equation}\label{choiceofsigma1}
	\bar{c_{2}}\sigma^{\beta_{2}} \leq 4^{-(n+4)}.
	\end{equation}
	Now that we have chosen $\sigma,$ we set 
	\begin{equation}\label{choiceofH11}
	H_{1} : = 10^{5n} \sigma^{-2n}.
	\end{equation}
	Note that $H_{1}$ depends only on $n,N, k,p,\gamma, L.$ 
	Now, we chose the radius $R_{1} > 0$ small enough such that we have 
	\begin{equation}\label{smallnessradius1}
	\omega (R_{1})  + \int_{0}^{2R_{1}} \omega(\varrho) \frac{d \varrho}{\varrho} \leq 
	\frac{\sigma^{2n}}{6^n10^6 c_{4}} .
	\end{equation}
	Note that $R_{1}$ depends on $n,k,p,\gamma, L$ \emph{and} $\omega(\cdot).$ Also, if $a(\cdot)$ is a constant function, the dependence on $\omega( \cdot )$ is redundant. With this, we have chosen all the relevant parameters.\smallskip 
	
	\emph{ \textbf{Step 2: Excess decay}} We now want to prove 
	\begin{claim}[excess decay estimate]
		If 
		\begin{align}\label{upper bound}
		\left( \fint_{B_{i}} \left\lvert dw \right\rvert^{p}\right)^{\frac{1}{p}} \leq \lambda  ,
		\end{align} 
		then 
		\begin{equation}\label{excess decay eqn}
		E_{2}(V(d w), B_{i+1}) \leq \frac{1}{4} E_{2}(V(d w), B_{i}) + \frac{2c_{4}\lambda^{\frac{p}{2}}}{\sigma^{n}}\omega\left( R_{i} \right). 
		\end{equation}
	\end{claim}
	By the oscillation decay estimates  \eqref{Vdvscaling} and the choice of $\sigma$ in \eqref{choiceofsigma1}, we have, 
	\begin{equation*}
	E_{2}(V(d v_{i}), B_{i+1}) \leq \frac{1}{4^{n+4}} E_{2}(V(d v_{i}), B_{i}).
	\end{equation*}
	Now using the property of the mean, triangle inequality and elementary estimates along with lemma \ref{secondcomparison1}, we have, 
	\begin{align*}
	E_{2}(V(d w), B_{i+1}) &\leq 2 E_{2}(V(d v_{i}), B_{i+1}) + 2\sigma^{-n} \left( \fint_{B_{i}} \lvert V(d v_{i}) - V(d w )\rvert^{2} \right)^{\frac{1}{2}} \notag \\
	&\stackrel{\eqref{l1omega},\eqref{upper bound}}{\leq} 2E_{2}(V(d v_{i}),B_{i+1}) + \sigma^{-n} c_{4}\omega\left(R_{i} \right) \lambda^{\frac{p}{2}} .
	\end{align*}
	Also similarly, 
	\begin{align*}
	E_{2}(V(d v_{i}),B_{i}) &\stackrel{\eqref{l1omega},\eqref{upper bound}}{\leq} 2E_{2}(V(d w),B_{i}) + c_{4}\omega\left(R_{i} \right) \lambda^{\frac{p}{2}} .
	\end{align*}
	Combining the last three estimates, we have \eqref{excess decay eqn}.\smallskip

	\emph{\textbf{Step 3: Control on composite quantities}} Now we want to show, by induction that 
	\begin{equation}\label{induction for m +E}
	m_{i}(V(dw)) + E_{2}(V(d w), B_{i}) \leq \lambda^{\frac{p}{2}}  \qquad \text{ for all } i \geq 1.
	\end{equation}
	This is true for $i = 1$ by our choice of $H_{1}$ in \eqref{choiceofH11} and elementary estimates. Thus we assume this is true for all $j \in \lbrace 1, \ldots, i\rbrace$ and prove it for $i+1.$ But we have , 
	\begin{align*}
	\fint_{B_{j}} \left\lvert dw \right\rvert^{p} &= \fint_{B_{j}} \left\lvert V(dw ) \right\rvert^{2} \leq \left[ m_{j}(V(dw))\right]^{2} + \left[ E_{2}(V(d w), B_{j})\right]^{2}\\ &\leq \left[ m_{j}(V(dw)) +  E_{2}(V(d w), B_{j})\right]^{2}. 
	\end{align*} Thus the induction hypotheses 
	implies the bound 
	\begin{align*}
	\left( \fint_{B_{j}} \left\lvert dw \right\rvert^{p}\right)^{\frac{1}{p}} \leq \lambda \text{ for all } j \in \lbrace 1, \ldots, i\rbrace. 
	\end{align*}
	Thus, by the excess decay estimate \eqref{excess decay eqn} for each $j$ and summing, we obtain, 
	\begin{align*}
	\sum_{j=2}^{i+1} E_{2}(V(d w), B_{j}) \leq \frac{1}{2} \sum_{j=1}^{i} E_{2}(V(d w), B_{j}) + \frac{2c_{4}\lambda^{\frac{p}{2}}}{\sigma^{n}} \sum_{j=1}^{i}  
	\omega\left( R_{j} \right).
	\end{align*}
	Note that 
	\begin{align*}
	\int_{0}^{2R} \omega(\varrho) \frac{d \varrho}{\varrho} &= \sum_{i=0}^{\infty} \int_{R_{i+1}}^{R_{i}} \omega(\varrho) \frac{d \varrho}{\varrho} 
	+ \int_{R}^{2R} \omega(\varrho) \frac{d \varrho}{\varrho} \\
	&\geq \sum_{i=0}^{\infty} \omega(R_{i+1}) \int_{R_{i+1}}^{R_{i}} \frac{d \varrho}{\varrho} + \omega(R) \int_{R}^{2R} \frac{d \varrho}{\varrho} \\
	&= \log(\frac{1}{\sigma}) \sum_{i=0}^{\infty} \omega(R_{i+1}) + \omega(R) \log 2  \geq \log 2\sum_{i=0}^{\infty} \omega(R_{i}). 
	\end{align*}
	
	Thus, we deduce, using the last estimate and \eqref{smallnessradius1}, 
	\begin{align*}
	\sum_{j=1}^{i+1} E_{2}(V(d w), B_{j}) \leq 2 E_{2}(V(d w), B_{1}) + \frac{2c_{4}\lambda^{\frac{p}{2}}}{\sigma^{n}} \sum_{j=1}^{i}\omega\left( R_{j} \right)
	\leq \frac{\sigma^{n}\lambda^{\frac{p}{2}}}{100}.
	\end{align*}
	Now, using this and writing the difference of averages as a telescoping sum, we deduce, 
	\begin{align*}
	m_{i+1}(V(dw)) -  m_{1}(V(dw)) &= \sum_{j=1}^{i} \left( m_{j+1}(V(dw)) -  m_{j}(V(dw))\right) \\
	&\leq \sum_{j=1}^{i} \fint_{B_{j+1}}\left\lvert V(dw) - (V(dw))_{B_{j}} \right\rvert \\
	&\leq \sigma^{-n}\sum_{j=1}^{i} E_{2}(V(d w), B_{j}) \leq \frac{\lambda^{\frac{p}{2}}}{100}. 
	\end{align*}
	The last two estimate together yields \eqref{induction for m +E}. \smallskip

	\emph{\textbf{Step 5:}} Finally, since $p \geq 2,$ by \eqref{induction for m +E}, we have, for $i \geq 1,$  
	\begin{align*}
	\left\lvert \left( \fint_{B_{i}} d w \right) \right\rvert^{\frac{p}{2}} &\leq   \fint_{B_{i}}\left\lvert dw \right\rvert^{\frac{p}{2}} =  \fint_{B_{i}}\left\lvert V(dw) \right\rvert \\ 
	&=  E_{2}(V(d w), B_{i}) + m_{i}(V(dw)) \leq \lambda^{\frac{p}{2}}.
	\end{align*}
	But this implies, if $x$ is a Lebesgue point of $du,$ then we have, 
	$$ \left\lvert dw(x)\right\rvert = \left\lvert \lim_{i \rightarrow \infty} \left( \fint_{B_{i}} d w \right) \right\rvert \leq \lambda. $$ This completes the proof. 
\end{proof}

\subsubsection{Proof of continuity}
\noindent Now we are in a position to prove Theorem \ref{dwcontinuityhomogeneousDini}.\smallskip  

\begin{proof}[Proof of theorem \ref{dwcontinuityhomogeneousDini}]  Now, we set $\lambda := \lVert d w \rVert_{L^{\infty}(\Omega)} + 1.$ To prove the continuity of $dw,$ it is obviously enough to prove 
	that $V(dw)$ is continuous. Thus the strategy of the proof is to show that 
	$V(dw)$ is the locally uniform limit of a net of continuous maps, defined by the averages 
	$$ x \mapsto \left( V(dw)\right)_{B(x, \rho)}.$$ To do this, we pick any $\Omega' \subset \subset \Omega$. we show that for every $x \in \Omega$ and every 
	$\varepsilon > 0,$ there exists a radius $$0 < r_{\varepsilon} \leq \operatorname*{dist}(\Omega', \partial\Omega)/1000 = R^{\ast} ,$$ depending only on 
	$n,k,p, \gamma, L, \omega( \cdot ), \varepsilon$ such that for every $x \in \Omega',$ the estimate 
	\begin{equation}\label{oscillationVdw}
	\left\lvert \left( V(dw)\right)_{B(x, \rho)} - \left( V(dw)\right)_{B(x, \varrho)} \right\rvert \leq \lambda^{\frac{p}{2}} \varepsilon \qquad \text{ holds for every } \rho, \varrho \in 
	(0, r_{\varepsilon}].
	\end{equation}
	This would imply that the sequence of maps $x \mapsto \left( V(dw)\right)_{B(x, \rho)}$ are uniformly Cauchy and would conclude the continuity of $V(dw).$ \smallskip 
	
	\emph{ \textbf{Step 1: Choice of constants}} We fix $\varepsilon >0.$ Now we choose the constants as in the proof of boundedness, but in the the scale $\varepsilon.$ 
	More precisely, we choose now, 
	we choose $\sigma \in (0,\frac{1}{4})$ small enough such that 
	\begin{equation}\label{choiceofsigma1epsilon}
	\bar{c_{2}}\sigma^{\beta_{2}} \leq \frac{\varepsilon}{4^{(n+4)}}.
	\end{equation}
	Now, we fix a radius $R_{\varepsilon} > 0$ small enough such that we have 
	\begin{equation}\label{smallnessradius1epsilon}
	\omega (R_{\varepsilon})  + \int_{0}^{2R_{\varepsilon}} \omega(\varrho) \frac{d \varrho}{\varrho} \leq 
	\frac{\sigma^{2n}\varepsilon}{6^n10^6 c_{4}} .
	\end{equation}
	Note that $R_{\varepsilon}$ depends on $n,k,p,\gamma, L$ ,$\omega(\cdot),$ and this time, also on $\varepsilon.$ Also, if $a(\cdot)$ is a constant function, the dependence on $\omega( \cdot )$ is redundant. 
	With this, we have chosen all the relevant parameters.\smallskip

	\emph{ \textbf{Step 2: Smallness of the excess}} Proceeding exactly as in Step 3 of the proof of theorem \ref{dwsupboundhomogeneousDini}, we obtain  
	\begin{claim}
		For any $i \geq 1,$ we have 
		\begin{equation}
		E_{2}(V(d w), B_{i+1}) \leq \frac{\varepsilon}{4} E_{2}(V(d w), B_{i}) + \frac{2c_{4}\lambda^{\frac{p}{2}}}{\sigma^{n}}\omega\left( R_{i} \right). 
		\end{equation}
	\end{claim}
	This now easily implies the following. 
	\begin{claim}
		Given $\varepsilon \in (0,1),$ there exists a positive radius $r_{\varepsilon} = r_{\varepsilon} ( n,k,p,\gamma, L$, $\omega(\cdot), \varepsilon )$ such that we have 
		\begin{equation}\label{smallness of the excess}
		E_{2}(V(d w), B_{\rho}) \leq \lambda^{\frac{p}{2}} \varepsilon, 
		\end{equation}
		whenever $0 < \rho \leq r_{\varepsilon}$ and $B_{\rho} \subset \subset \Omega.$
	\end{claim}
	\emph{ \textbf{Step 3: Cauchy estimate}}\smallskip 
	
	\noindent Now we finish the proof. First we fix a radius one last time. 
	Given $\varepsilon \in (0, 1),$ using \eqref{smallness of the excess} we choose a radius $R_{3} \equiv R_{3}(n,k,p,\gamma, L$, $\omega(\cdot), \varepsilon )$ such that 
	\begin{equation}
	\sup_{0 < \rho \leq R_{3}} \sup_{x \in \Omega_{0}} E_{2}(V(d w), B_{\rho}) \leq \frac{\sigma^{4n}\lambda^{\frac{p}{2}} \varepsilon}{10^{10}}.
	\end{equation}
	We set $$ R_{0} := \frac{\min{\lbrace R^{\ast}, R_{3}, R_{1} \rbrace} }{16}.$$
	Now we want to show that given two integers $2 \leq i_{1} < i_{2},$ we have the estimate 
	\begin{equation}\label{integer ball average oscillation}
	\left\lvert (V(dw))_{B_{i_{1}}} - (V(dw))_{B_{i_{2}}} \right\rvert \leq \frac{\lambda^{\frac{p}{2}} \varepsilon}{10}.
	\end{equation}
	Note that this will complete the proof, since for any $0 < \rho < \varrho \leq r_{\varepsilon},$ there exist integers such that 
	$$ \sigma^{i_{1}+1}R_{0} < \varrho \leq \sigma^{i_{1}}R_{0} \quad \text{ and } \sigma^{i_{2}+1}R_{0} < \rho \leq \sigma^{i_{2}}R_{0}. $$
	Also, we have the easy estimates
	\begin{align*}
	\left\lvert (V(dw))_{B_{\varrho}} - (V(dw))_{B_{i_{1}+1}} \right\rvert &\leq \fint_{B_{i_{1}+1}} \left\lvert  V(dw) - (V(dw))_{B_{\varrho}}  \right\rvert \\
	&\leq \frac{\lvert B_{\varrho}\rvert}{\lvert B_{i_{1}+1}\rvert} \fint_{B_{\varrho}} \left\lvert  V(dw) - (V(dw))_{B_{\varrho}}  \right\rvert \\
	&\leq \sigma^{-n} E_{2}(V(d w), B_{\varrho}) \leq  \frac{\lambda^{\frac{p}{2}} \varepsilon}{1000}
	\end{align*}
	and similarly 
	$$  \left\lvert (V(dw))_{B_{\rho}} - (V(dw))_{B_{i_{2}+1}} \right\rvert \leq \frac{\lambda^{\frac{p}{2}} \varepsilon}{1000}.$$
	These two estimates combined with \eqref{integer ball average oscillation} will establish \eqref{oscillationVdw}. Thus it only remains to establish \eqref{integer ball average oscillation}. But summing up \eqref{smallness of the excess} for 
	$i = i_{1}-1$ to $i_{2}-2,$ we obtain 
	$$\sum_{i=i_{1}}^{i_{2}-1}E_{2}(V(d w), B_{i}) \leq E_{2}(V(d w), B_{i_{1}-1}) + \frac{2c_{4}\lambda^{\frac{p}{2}}}{\sigma^{n}}\sum_{i=1}^{\infty} \omega\left( R_{i} \right)
	\leq \frac{ \sigma^{2n} \lambda^{\frac{p}{2}} \varepsilon}{50}. $$
	This yields \eqref{integer ball average oscillation} via the elementary estimate 
	\begin{align*}
	\left\lvert (V(dw))_{B_{i_{1}}} - (V(dw))_{B_{i_{2}}} \right\rvert &\leq \sum_{i=i_{1}}^{i_{2}-1} \fint_{B_{i+1}} \left\lvert V(dw) - (V(dw))_{B_{i}} \right\rvert \\
	&\leq \sigma^{-n} \sum_{i=i_{1}}^{i_{2}-1}E_{2}(V(d w), B_{i}) .
	\end{align*}
	This establishes \eqref{oscillationVdw}.\smallskip 
	
	\emph{ \textbf{Step 4: Final conclusions}} Now we prove the conclusions of theorem \ref{dwcontinuityhomogeneousDini}. Note that \eqref{oscillationVdw} implies $V(dw)$ is continuous and thus so is $dw.$ The estimate 
	in part (i) follows from the pointwise estimate in theorem \ref{dwsupboundhomogeneousDini} by standard covering arguments and interpolation arguments to lower the exponent. The conclusion of 
	part (ii) follows from \eqref{oscillationVdw} by exactly the same arguments as in the proof of Theorem 2 in \cite{KuusiMingione_nonlinearStein}, just replacing $Dw_{j}$ by $dw.$ 
\end{proof}

\subsection{Gauge fixing}
Now we show that if we are interested only in $du,$ we can always assume that $u$ solving \eqref{quasilinear maxwell tangential intro} is also coclosed. Moreover, given a solution $u,$ there exists a coclosed solution upto the addition of a closed form.  
\begin{lemma}[Gauge fixing lemma]\label{gauge fixing lemma}
	Let $u \in W^{1,p}_{loc}\left(\Omega; \Lambda^{k}\mathbb{R}^{n}\otimes\mathbb{R}^{N} \right)$ be a local solution to system \eqref{quasilinear maxwell tangential intro}. Then for any $B_{R} \subset \subset \Omega,$ there exists a coclosed form $\tilde{u} \in W^{1,p}\left(B_{R}; \Lambda^{k}\mathbb{R}^{n}\otimes\mathbb{R}^{N} \right)$ which is also a local solution to system \eqref{quasilinear maxwell tangential intro} and we have $$ \delta \tilde{u} = 0 \quad \text{ and } \quad  d\tilde{u}= du \qquad \text{ a.e. in } B_{R}. $$
\end{lemma}
\begin{proof}
	Since $u \in W^{1,p}_{loc},$ we have $\delta u \in L^{p}\left(B_{R}; \Lambda^{k-1}\mathbb{R}^{n}\otimes\mathbb{R}^{N} \right).$ Thus we can find $\theta \in W^{2,p} \left(B_{R}; \Lambda^{k}(\mathbb{R}^{n};\mathbb{R}^{N}) \right)$ such that 
	\begin{align}\label{gauge fixing system}
	\left\lbrace \begin{aligned} 	\delta d \theta &= \delta u &&\text{ in } B_{R} \\
	\delta \theta &= 0  &&\text{ in } B_{R} \\
	\nu \wedge \theta &= 0  &&\text{ on } \partial B_{R}. \end{aligned} \right.
	\end{align}
	Now the result follows by setting $\tilde{u} = u - d\theta.$ 
\end{proof}
\subsection{Preparatory estimates}
\subsubsection{General setting}
Let $x_{0} \in \Omega$ and $0 < r <1 $ be such that $B(x_{0}, 2r) \subset \Omega.$ By lemma \ref{gauge fixing lemma}, it is enough to consider the system
\begin{align}\label{quasilinear maxwell local coulomb}
\left\lbrace \begin{aligned}
\delta ( a(x) \lvert du \rvert^{p-2} du) )  &= f   &&\text{ in } B_{2r},\\
\delta u &= 0 &&\text{ in } B_{2r}.
\end{aligned} 
\right. 
\end{align}
For $j\geq 0,$ we set 
$$ B_{j}:= B(x_{0}, r_{j}), \qquad r_{j}:= \sigma^{j}r, \quad \sigma \in (0,1/4).$$

\noindent For $j \geq 0,$ we define $w_{j} \in u + W_{\delta,T}^{1,p}\left( B_{j}; \Lambda^{k}\mathbb{R}^{n}\otimes \mathbb{R}^{N}\right)$ to be the unique solution of 
\begin{align}\label{quasilinear maxwell tangential homogeneous general}
\left\lbrace \begin{aligned}
\delta ( a(x) \lvert dw_{j} \rvert^{p-2} dw_{j}) )  &= 0   &&\text{ in } B_{j},\\
\delta w_{j} &= 0 &&\text{ in } B_{j}, \\
\nu\wedge w_{j} &= \nu\wedge u &&\text{  on } \partial B_{j},
\end{aligned} 
\right. 
\end{align}
and $v_{j} \in w_{j} + W_{\delta,T}^{1,p}\left( \frac{1}{2}B_{j}; \Lambda^{k}\mathbb{R}^{n}\otimes \mathbb{R}^{N}\right)$ to be the unique solution of 
\begin{align}\label{quasilinear maxwell tangential homogeneous frozen general}
\left\lbrace \begin{aligned}
\delta ( a(x_{0}) \lvert dv_{j} \rvert^{p-2} dv_{j}) )  &= 0   &&\text{ in } \frac{1}{2} B_{j},\\
\delta v_{j} &= 0 &&\text{ in } \frac{1}{2}B_{j}, \\
\nu\wedge v_{j} &= \nu\wedge w_{j} &&\text{  on } \partial \left( \frac{1}{2}B_{j}\right).
\end{aligned} 
\right. 
\end{align}
We set 
\begin{equation}\label{qdef}
q := \left\lbrace \begin{aligned}
&\frac{np}{n+p} &&\text{if } n >2 \text{ or } p>2 \\
&\frac{3}{2} &&\text{if } n= p = 2 \\
&\frac{np}{np -n +p} &&\text{if } 1 < p < 2 
\end{aligned}\right. \end{equation}
and set 
\begin{equation}\label{sdef}
s := \left\lbrace \begin{aligned}
&p^{'} &&\text{if } p \geq 2 \\
&p &&\text{if } 1 <  p < 2. 
\end{aligned}\right. \end{equation}
Note that if $n > 2$ or $p > 2$, our choice implies 
\begin{equation}\label{prop q}
q \geq \left( p^{\ast}\right)^{'} \quad \text{ if } p < n \qquad \text{ and } \qquad q' = \left( p'\right)^{\ast}. 
\end{equation}

\subsubsection{Comparison estimates}
\begin{lemma}\label{u w comparison any p easy}
	Let $u$ be as in \eqref{quasilinear maxwell local coulomb} and $w_{j}$ be as in \eqref{quasilinear maxwell tangential homogeneous general} and $j \geq 0.$ There exists a constant 
	$c_{5} \equiv c_{5}\left( n, k, p, \gamma, L, \right)$ such that the following inequality 
	\begin{equation}\label{estimateuwforanyp}
	\fint_{B_{j}}\left( \left\lvert du \right\rvert + \left\lvert dw_{j} \right\rvert  \right)^{p-2}\left\lvert du  - dw_{j}\right\rvert^{2}  
	\leq c_{5} \fint_{B_{j}} \lvert f \rvert \lvert u -w_{j}\rvert .
	\end{equation}
	holds for any $p >1.$
\end{lemma}
\begin{proof}
	Weak formulation of \eqref{quasilinear maxwell local coulomb} and \eqref{quasilinear maxwell tangential homogeneous general} gives, 
	$$ \fint_{B_{j}} a(x)\left\langle \lvert du \rvert^{p-2}du - \lvert dw_{j} \rvert^{p-2}dw_{j}; d\phi\right\rangle = \fint_{B_{j}} \left\langle f; \phi\right\rangle , $$ 
	for all $ \phi \in W^{1,p}_{\delta,T}\left( B_{j};\Lambda^{k}\mathbb{R}^{n}\otimes \mathbb{R}^{N} \right). $
	Substituting $u -w_{j}$ in place of $\phi,$ we obtain, 
	$$ \fint_{B_{j}} a(x)\left\langle \lvert du \rvert^{p-2}du - \lvert dw_{j} \rvert^{p-2}dw_{j}; du - dw_{j}\right\rangle = \fint_{B_{j}} \left\langle f; u-w_{j}\right\rangle . $$
	By classical monotonicity estimate \eqref{monotonicity}, we obtain, 
	\begin{multline*}
	\fint_{B_{j}}\left( \left\lvert du \right\rvert + \left\lvert dw_{j} \right\rvert  \right)^{p-2}\left\lvert du  - dw_{j}\right\rvert^{2} \\ 
	\leq c \fint_{B_{j}} a(x)\left\langle \lvert du \rvert^{p-2}du - \lvert dw_{j} \rvert^{p-2}dw_{j}; du - dw_{j}\right\rangle .
	\end{multline*}
	This proves \eqref{estimateuwforanyp}.  \end{proof}\smallskip 

\noindent Proceeding similarly to Lemma \ref{secondcomparison1}, we deduce 
\begin{lemma}\label{secondcomparison2}
	Let $w,v_{i}$ be as before and  $j \geq 0.$ Then there exists a constant $c \equiv c\left( n, N, k, p, \gamma, L \right)$ such that we have the inequality 
	\begin{align}\label{l1omegageneralV}
	 \fint_{\frac{1}{2}B_{j}} \lvert V\left( dv_{j} \right)  - V \left( dw_{j} \right)  \rvert^{2}   &\leq c\left[ \omega\left(r_{j}\right) \right]^{2}  
	 \fint_{\frac{1}{2} B_{j}} \lvert dw_{j} \rvert^{p}.
	\end{align}
	\end{lemma}

\paragraph{Comparison estimates for $p >2$ case}
Using \eqref{constant cv}, lemma \ref{secondcomparison2} immediately implies the following. 
\begin{lemma}\label{secondcomparisonpgeq2}
	Let $w_{j}$ be as in \eqref{quasilinear maxwell tangential homogeneous general}, $v_{j}$ be as in \eqref{quasilinear maxwell tangential homogeneous frozen general} with $p >2 $ and $j \geq 0.$ Then there exists a constant $c_{4} \equiv c_{4}\left( n, N, k, p, \gamma, L \right)$ such that we have the inequality
	\begin{align}\label{l1omegageneral}
	\left( \fint_{\frac{1}{2}B_{j}} \lvert dv_{j} - dw_{j} \rvert^{p} \right)^{\frac{1}{p}}  &\leq c_{4}\left[ \omega\left(r_{j}\right) \right]^{\frac{2}{p}}  
	\left( \fint_{\frac{1}{2} B_{j}} \lvert dw_{j} \rvert^{p}\right)^{\frac{1}{p}}.
	\end{align}
	\end{lemma}

\begin{lemma}\label{firstcomparison}
	Let $u$ be as in \eqref{quasilinear maxwell local coulomb} with $p >2 $ and $w_{j}$ be as in \eqref{quasilinear maxwell tangential homogeneous general} and $j \geq 0.$ There exists a constant
	$c_{6} \equiv c_{6}\left( n, k, p, \gamma, L, \right)$
	such that the following inequality
	\begin{equation}\label{fq u and w}
	\left( \fint_{B_{j}} \lvert d u - d w_{j}\rvert^{p} \right)^{\frac{1}{p}}  \leq c_{6}
	\left( r_{j}^{q} \fint_{B_{j}} \lvert f \rvert^{q} \right)^{\frac{1}{q(p-1)}},
	\end{equation}
	holds for every $q \geq ( p^{*})^{'} $ when $p<n$ and for every $q >1$ when $ p \geq n.$ Moreover, when $j \geq 1,$ there exists another constant 
	$c_{7} \equiv c_{7}\left( n, k, p, \gamma, L, \right)$ such that the following inequality holds 
	\begin{equation}\label{fq j j-1}
	\left( \fint_{B_{j}} \lvert d w_{j-1} - d w_{j}\rvert^{p} \right)^{\frac{1}{p}}  \leq c_{7}
	\left( r_{j-1}^{q} \fint_{B_{j-1}} \lvert f \rvert^{q} \right)^{\frac{1}{q(p-1)}}
	\end{equation}
\end{lemma}
\begin{proof}
	Since $p >2,$ \eqref{estimateuwforanyp} yields, 
	\begin{equation}\label{estimateuwforpbigger2}
	\fint_{B_{j}}\lvert du  - dw_{j}\rvert^{p}  \leq c \fint_{B_{j}} \lvert f \rvert \lvert u -w_{j}\rvert .
	\end{equation}
	Since $u - w_{j} \in W^{1,p}_{\delta,T}\left( B_{j};\Lambda^{k}\mathbb{R}^{n}\otimes \mathbb{R}^{N} \right),$ \eqref{estimateuwforpbigger2} implies, by H\"{o}lder and Poincar\'{e}-Sobolev inequality ( cf. Proposition \ref{poincaresobolev}), 
	\begin{align*}
	\fint_{B_{j}}\lvert du  - dw_{j}\rvert^{p}  \leq \fint_{B_{j}} \lvert f \rvert \lvert u -w_{j}\rvert &\leq c \left( \fint_{B_{j}} \lvert u -w_{j}\rvert^{p^{*}} \right)^{\frac{1}{p^{*}}} 
	\left( \fint_{B_{j}} \lvert f \rvert \right)^{\frac{1}{q}} 
	\\&\leq c r_{j}\left( \fint_{B_{j}} \lvert d u - d w_{j}\rvert^{p} \right)^{\frac{1}{p}}\left( \fint_{B_{j}} \lvert f \rvert \right)^{\frac{1}{q}}.
	\end{align*}
	This proves \eqref{fq u and w}. Applying \eqref{fq u and w} on $B_{j-1}$ and $B_{j},$ we deduce 
	\begin{align*}
	\fint_{B_{j}} \lvert d w_{j-1} - d w_{j}\rvert^{p} &\leq c\sigma^{-n}\fint_{B_{j-1}}\lvert du  - dw_{j}\rvert^{p} + c \fint_{B_{j}}\lvert du  - dw_{j}\rvert^{p} \\
	&\leq c \sigma^{-n - \frac{p(n-q)}{q(p-1)}}  \left( r_{j-1}^{q} \fint_{B_{j-1}} \lvert f \rvert^{q} \right)^{\frac{p}{q(p-1)}}.
	\end{align*}
	This proves \eqref{fq j j-1}. 
\end{proof}

\begin{lemma}\label{finalcomparispbigger2}
	Let $u$ be as in \eqref{quasilinear maxwell local coulomb} with $p >2 $ and $v_{j},w_{j}$ be as before in 
	\eqref{quasilinear maxwell tangential homogeneous general}, \eqref{quasilinear maxwell tangential homogeneous frozen general}, respectively, for $j \geq 1$ and let $q$ be as in 
	\eqref{qdef}. Suppose for some $\lambda >0$ we have  
	\begin{equation}\label{upper bound of f}
	r_{j-1} \left( \fint_{B_{j-1}}\lvert f \rvert^{q} \right)^{\frac{1}{q}} \leq \lambda^{p-1}
	\end{equation}
	and for some constant $A \geq 1,$ the estimates 
	\begin{equation}\label{both side bounds on w}
	\sup\limits_{\frac{1}{2}B_{j}} \left\lvert dw_{j} \right\rvert \leq A\lambda \qquad \text{ and } \qquad \frac{\lambda}{A} \leq \left\lvert dw_{j-1} \right\rvert 
	\leq A\lambda \quad \text{ in } B_{j}
	\end{equation}
	hold. Then there exists a constant $c_{8} \equiv c_{8}\left( n, k, p, \gamma, L, \right)$such that 
	\begin{align}\label{grad u -grad v pbigger2}
	\left( \fint_{\frac{1}{2}B_{j}} \left\lvert d u - d v_{j} \right\rvert^{p'}\right)^{\frac{1}{p'}} \leq c_{8}\omega(r_{j})\lambda 
	+ c_{8}\lambda^{2-p}r_{j-1}\left( \fint_{B_{j-1}} \left\lvert f \right\rvert^{q} \right)^{\frac{1}{q}}.
	\end{align}
\end{lemma}
\begin{proof}
	We first show the estimate 
	\begin{align}\label{grad u -grad w pbigger2}
	\left( \fint_{B_{j}} \left\lvert d u - d w_{j} \right\rvert^{p'}\right)^{\frac{1}{p'}} \leq c\lambda^{2-p}r_{j-1}\left( \fint_{B_{j-1}} \left\lvert f \right\rvert^{q} \right)^{\frac{1}{q}}.
	\end{align}
	Next we would establish 
	\begin{align}\label{grad v -grad w pbigger2}
	\left( \fint_{\frac{1}{2}B_{j}} \left\lvert d w_{j} - d v_{j} \right\rvert^{p'}\right)^{\frac{1}{p'}} \leq c\omega(r_{j})\lambda 
	+ c\lambda^{2-p}r_{j-1}\left( \fint_{B_{j-1}} \left\lvert f \right\rvert^{q} \right)^{\frac{1}{q}}.
	\end{align}
	Then \eqref{grad u -grad v pbigger2} would follow by combining \eqref{grad u -grad w pbigger2} and \eqref{grad v -grad w pbigger2}. 
	
	\textbf{Step 1: Proof of \eqref{grad u -grad w pbigger2}}
	
	We have
	\begin{align*}
	&\left( \fint_{B_{j}} \left\lvert d u - d w_{j} \right\rvert^{p'}\right)^{\frac{1}{p'}} \\
	&\stackrel{\eqref{both side bounds on w}}{\leq} 
	\left( \frac{A}{\lambda} \right)^{p-2}  \left( \fint_{B_{j}} \left\lvert dw_{j-1} \right\rvert^{p'(p-2)}\left\lvert d u - d w_{j} \right\rvert^{p'}\right)^{\frac{1}{p'}} \\
	&\begin{aligned}
	\leq c\left( \frac{A}{\lambda} \right)^{p-2} &\left( \fint_{B_{j}} \left\lvert dw_{j} - dw_{j-1} \right\rvert^{p'(p-2)}\left\lvert d u - d w_{j} \right\rvert^{p'}\right)^{\frac{1}{p'}}
	\\ &\qquad + c\left( \frac{A}{\lambda} \right)^{p-2} \left( \fint_{B_{j}} \left\lvert dw_{j} \right\rvert^{p'(p-2)}\left\lvert d u - d w_{j} \right\rvert^{p'}\right)^{\frac{1}{p'}}
	= I_{1} + I_{2}.
	\end{aligned}
	\end{align*}
	Now we estimate the last two integrals separately 
	We have, noting that $p' \leq p,$
	\begin{align*}
	I_{1} &\leq c\left( \frac{A}{\lambda} \right)^{p-2}
	\left( \fint_{B_{j}} \left\lvert dw_{j} - dw_{j-1} \right\rvert^{p}\right)^{\frac{p-2}{p}}
	\left( \fint_{B_{j}} \left\lvert d u - d w_{j} \right\rvert^{p}\right)^{\frac{1}{p}} \\
	&\stackrel{\eqref{fq j j-1},\eqref{fq u and w}}{\leq} c\lambda^{2-p}r_{j-1}\left( \fint_{B_{j-1}} \left\lvert f \right\rvert^{q} \right)^{\frac{1}{q}}.
	\end{align*}
	For the last integral, we note that 
	\begin{align*}
	\left( \fint_{B_{j}} \left\lvert dw_{j} \right\rvert^{p'(p-2)}\left\lvert d u - d w_{j} \right\rvert^{p'}\right)^{\frac{1}{p'}} 
	&= \left( \fint_{B_{j}} \left\lvert dw_{j} \right\rvert^{p-2}\left\lvert d u - d w_{j} \right\rvert^{p'} \left\lvert dw_{j} \right\rvert^{\frac{p-2}{p-1}}\right)^{\frac{1}{p'}} \\
	&\stackrel{\text{H\"{o}lder}}{\leq} 
	 \left( \fint_{B_{j}} \left\lvert dw_{j} \right\rvert^{p-2}\left\lvert d u - d w_{j} \right\rvert^{2} \right)^{\frac{1}{2}}
	\left( \fint_{B_{j}} \left\lvert dw_{j} \right\rvert^{p}\right)^{\frac{p-2}{2p}} .\end{align*}
	We also have 
	\begin{align*}
	\fint_{B_{j}} \left\lvert dw_{j} \right\rvert^{p}
	&\leq  c \fint_{B_{j}} \left\lvert dw_{j} - dw_{j-1} \right\rvert^{p} + c \fint_{B_{j}} \left\lvert dw_{j-1} \right\rvert^{p} \\
	&\stackrel{\eqref{fq j j-1},\eqref{both side bounds on w}}{\leq} c \left( r_{j}^{q} \fint_{B_{j}} \lvert f \rvert^{q} \right)^{\frac{p}{q(p-1)}} 
	+ cA\lambda^{p} \stackrel{\eqref{upper bound of f}}{\leq} c\lambda^{p}.
	\end{align*}
	Thus combining the two estimates, we deduce, 
	\begin{align*}
	I_{2}
	&\leq c\lambda^{\frac{2-p}{2}} \left( \fint_{B_{j}} \left\lvert dw_{j} \right\rvert^{p-2}\left\lvert d u - d w_{j} \right\rvert^{2} \right)^{\frac{1}{2}} \\
	&\leq c\lambda^{\frac{2-p}{2}} 
	\left( \fint_{B_{j}} \left( \left\lvert du \right\rvert + \left\lvert dw_{j} \right\rvert \right)^{p-2}\left\lvert d u - d w_{j} \right\rvert^{2} \right)^{\frac{1}{2}} \\
	&\stackrel{\eqref{estimateuwforanyp}}{\leq} c \lambda^{\frac{2-p}{2}}  \left( \fint_{B_{j}} \lvert f \rvert \lvert u -w_{j}\rvert\right)^{\frac{1}{2}} \\
	&\leq c\lambda^{\frac{2-p}{2}} \left( \fint_{B_{j}} \lvert u -w_{j}\rvert^{q'}\right)^{\frac{1}{2q'}} \left( r_{j}^{q} \fint_{B_{j}} \lvert f \rvert^{q} \right)^{\frac{1}{2q}} \\
	&\stackrel{\eqref{prop q}}{\leq} c\lambda^{\frac{2-p}{2}} \left( \fint_{B_{j}} \lvert du - dw_{j}\rvert^{p'}\right)^{\frac{1}{2p'}} \left( r_{j}^{q} \fint_{B_{j}} \lvert f \rvert^{q} \right)^{\frac{1}{2q}} \\ 
	&= c \left[ \left( \fint_{B_{j}} \lvert du - dw_{j}\rvert^{p'}\right)^{\frac{1}{p'}}\right]^{\frac{1}{2}}
	\left[ \lambda^{2-p} \left( r_{j}^{q} \fint_{B_{j}} \lvert f \rvert^{q} \right)^{\frac{1}{q}} \right]^{\frac{1}{2}}
	\end{align*}
	But this implies, by Young's inequality with $\varepsilon > 0, $
	\begin{align*}
	I_{2} 
	 	&\leq  \varepsilon \left( \fint_{B_{j}} \lvert du - dw_{j}\rvert^{p'}\right)^{\frac{1}{p'}} + c\lambda^{2-p} \left( r_{j}^{q} \fint_{B_{j}} \lvert f \rvert^{q} \right)^{\frac{1}{q}} \\
	&\leq \varepsilon \left( \fint_{B_{j}} \lvert du - dw_{j}\rvert^{p'}\right)^{\frac{1}{p'}} + c\lambda^{2-p} \left( r_{j-1}^{q} \fint_{B_{j-1}} \lvert f \rvert^{q} \right)^{\frac{1}{q}}.
	\end{align*}
	Combining these estimates and choosing $\varepsilon > 0$ small enough, we obtain \eqref{grad u -grad w pbigger2}. \bigskip 
	
	\textbf{Step 2: Proof of \eqref{grad v -grad w pbigger2})}
	\noindent We have
	\begin{align*}
	&\left( \fint_{\frac{1}{2}B_{j}} \left\lvert d v_{j} - d w_{j} \right\rvert^{p'}\right)^{\frac{1}{p'}} \\
	&\stackrel{\eqref{both side bounds on w}}{\leq} 
	\left( \frac{A}{\lambda} \right)^{p-2}  \left( \fint_{\frac{1}{2}B_{j}} \left\lvert dw_{j-1} \right\rvert^{p'(p-2)}\left\lvert d v_{j} - d w_{j} \right\rvert^{p'}\right)^{\frac{1}{p'}} \\
	&\begin{aligned}
	\leq c\left( \frac{A}{\lambda} \right)^{p-2} &\left( \fint_{\frac{1}{2}B_{j}} \left\lvert dw_{j} - dw_{j-1} \right\rvert^{p'(p-2)}\left\lvert d v_{j} - d w_{j} \right\rvert^{p'}\right)^{\frac{1}{p'}}
	\\ &\qquad + c\left( \frac{A}{\lambda} \right)^{p-2} \left( \fint_{\frac{1}{2}B_{j}} \left\lvert dw_{j} \right\rvert^{p'(p-2)}\left\lvert d v_{j} - d w_{j} \right\rvert^{p'}\right)^{\frac{1}{p'}}.
	\end{aligned}
	\end{align*}
	The first integral on the right can be estimated easily as 
	\begin{align*}
	\left( \frac{A}{\lambda} \right)^{p-2}  &\left( \fint_{\frac{1}{2}B_{j}} \left\lvert dw_{j} - dw_{j-1} \right\rvert^{p'(p-2)}\left\lvert d v_{j} - d w_{j} \right\rvert^{p'}\right)^{\frac{1}{p'}} \\
	&\leq \left( \frac{A}{\lambda} \right)^{p-2}  \left( \fint_{\frac{1}{2}B_{j}} \left\lvert dw_{j} - dw_{j-1} \right\rvert^{p}\right)^{\frac{p-2}{p}}
	\left( \fint_{\frac{1}{2}B_{j}} \left\lvert d v_{j} - d w_{j} \right\rvert^{p}\right)^{\frac{1}{p}} \\
	&\stackrel{\eqref{l1omegageneral},\eqref{both side bounds on w}}{\leq}  
	\left( \frac{A}{\lambda} \right)^{p-2}  \left( \fint_{\frac{1}{2}B_{j}} \left\lvert dw_{j} - dw_{j-1} \right\rvert^{p}\right)^{\frac{p-2}{p}}
	\left[ \omega(r_{j})\right]^{\frac{2}{p}}\lambda \\
	&\leq c\omega(r_{j})\lambda + c \lambda^{1-p} \fint_{\frac{1}{2}B_{j}} \left\lvert dw_{j} - dw_{j-1} \right\rvert^{p} \\
	&\stackrel{\eqref{fq j j-1}}{\leq} c\omega(r_{j})\lambda + c \lambda^{1-p}\left( r_{j-1}^{q} \fint_{B_{j-1}} \lvert f \rvert^{q} \right)^{\frac{p}{q(p-1)}} \\
	&\stackrel{\eqref{upper bound of f}}{\leq} c\omega(r_{j})\lambda + c \lambda^{2-p}r_{j-1} \left( \fint_{B_{j-1}} \lvert f \rvert^{q} \right)^{\frac{1}{q}}.
	\end{align*}
	Now for the second integral we estimate as follows. 
	\begin{align*}
	&c\left( \frac{A}{\lambda} \right)^{p-2} \left( \fint_{\frac{1}{2}B_{j}} \left\lvert dw_{j} \right\rvert^{p'(p-2)}\left\lvert d v_{j} - d w_{j} \right\rvert^{p'}\right)^{\frac{1}{p'}} \\
	&\stackrel{\eqref{both side bounds on w}}{\leq} c\left( \frac{A}{\lambda} \right)^{p-2} \left( A\lambda \right)^{\frac{p-2}{2}} 
	\left( \fint_{\frac{1}{2}B_{j}} \left\lvert dw_{j} \right\rvert^{\frac{p'(p-2)}{2}}\left\lvert d v_{j} - d w_{j} \right\rvert^{p'}\right)^{\frac{1}{p'}} \\
	&\leq c \lambda^{\frac{2-p}{2}} \left( \fint_{\frac{1}{2}B_{j}} \left\lvert dw_{j} \right\rvert^{p-2}\left\lvert d v_{j} - d w_{j} \right\rvert^{2}\right)^{\frac{1}{2}} \\
	&\leq c \lambda^{\frac{2-p}{2}} \left( \fint_{\frac{1}{2}B_{j}} \left( \left\lvert dv_{j} \right\rvert + \left\lvert dw_{j} \right\rvert \right)^{p-2}
	\left\lvert d v_{j} - d w_{j} \right\rvert^{2}\right)^{\frac{1}{2}} \\
	&\leq  c \lambda^{\frac{2-p}{2}} \left( \fint_{\frac{1}{2}B_{j}} \left\lvert V(d v_{j}) - V(d w_{j}) \right\rvert^{2}\right)^{\frac{1}{2}} \\
	&\stackrel{\eqref{l1omega}, \eqref{both side bounds on w}}{\leq} c\omega(r_{j})\lambda . 
	\end{align*}
	Combining the estimates proves \eqref{grad v -grad w pbigger2} and finishes the proof of the lemma. 
\end{proof}

\paragraph{Comparison estimates for $ 1 < p \leq 2$}
\begin{lemma}\label{finalcomparispsmaller2}
	Let $u$ be as in theorem \ref{main theorem} with $1 < p \leq 2 $ and $v_{j},w_{j}$ be as before in 
	\eqref{quasilinear maxwell tangential homogeneous general}, \eqref{quasilinear maxwell tangential homogeneous frozen general}, respectively, for $j \geq 1$ and let $q$ be as in 
	\eqref{qdef}. Suppose for some $\lambda >0$ we have  
	\begin{align}\label{f bound p less 2}
	r_{j} \left( \fint_{B_{j}}\lvert f \rvert^{q} \right)^{\frac{1}{q}} \leq \lambda^{p-1} \\
	\intertext{ and }
	\left( \fint_{B_{j}}\lvert du \rvert^{p} \right)^{\frac{1}{p}} \leq \lambda \label{du bound p less 2}
	\end{align}
	hold. Then there exists a constant $c_{8} \equiv c_{8}\left( n, k, p, \gamma, L, \right)$such that 
	\begin{equation}\label{lambdafirstpsmaller2d}
	\left( \fint_{\frac{1}{2}B_{j}} \left\lvert d u - d v_{j} \right\rvert^{p}\right)^{\frac{1}{p}} \leq c_{8}\omega(r_{j})\lambda 
	+ c_{8}\lambda^{2-p}r_{j}\left( \fint_{B_{j}} \left\lvert f \right\rvert^{q} \right)^{\frac{1}{q}}. 
	\end{equation}
	Moreover, there exists another constant $c_{9} =c_{9}(n,k, p, \nu, L, \sigma)$ such that  
	\begin{equation}\label{lambdasecondpsmaller2}
	\left( \fint_{B_{j}} \left\lvert d u - d w_{j} \right\rvert^{p}\right)^{\frac{1}{p}} \leq 
	c_{9}\lambda^{2-p}r_{j}\left( \fint_{B_{j}} \left\lvert f \right\rvert^{q} \right)^{\frac{1}{q}}.
	\end{equation} 
\end{lemma}
\begin{proof} This follows exactly as in Lemma 7 of \cite{KuusiMingione_nonlinearStein}. 
	We first prove \eqref{lambdasecondpsmaller2}. By \eqref{v estimate p less 2} and  H\"{o}lder inequality, we deduce, 
	\begin{multline*}
	\left( \fint_{B_{j}} \left\lvert d u - d w_{j} \right\rvert^{p}\right)^{\frac{1}{p}} \leq  
	c\left( \fint_{B_{j}} \left\lvert V(du) - V(d w_{j}) \right\rvert^{2}\right)^{\frac{1}{p}} \\
	+   c\left( \fint_{B_{j}} \left\lvert V(du) - V(d w_{j}) \right\rvert^{2}\right)^{\frac{1}{2}}
	\left( \fint_{B_{j}} \left\lvert d u \right\rvert^{p}\right)^{\frac{(2-p)p}{2}}= I_{1} + I_{2} . 
	\end{multline*}
	Now we estimate the two terms on the right hand side. For the first one, using Young's' inequality with $\varepsilon > 0,$ we deduce, 
	\begin{equation*}
	I_{1} \leq \left(\fint_{B_{j}}\left( \left\lvert du \right\rvert + \left\lvert dw_{j} \right\rvert  \right)^{p-2}
	\left\lvert du  - dw_{j}\right\rvert^{2} \right)^{\frac{1}{p}}  \stackrel{\eqref{estimateuwforanyp}}{\leq} c \left(\fint_{B_{j}} \lvert f \rvert \lvert u -w_{j}\rvert \right)^{\frac{1}{p}}.
	\end{equation*}
	But we have, using H\"{o}lder inequality,  \begin{align*}
	\left(\fint_{B_{j}} \lvert f \rvert \lvert u -w_{j}\rvert \right)^{\frac{1}{p}} 	&\stackrel{ \eqref{qdef}}{\leq}  \left( \fint_{B_{j}} \lvert u -w_{j}\rvert^{q'} \right)^{\frac{1}{pq'}} 
	\left( \fint_{B_{j}} \lvert f \rvert^{q} \right)^{\frac{1}{pq}} \\
	&\stackrel{\eqref{poincaresobolevineq}}{\leq } c \left( \fint_{B_{j}} \lvert d u - d w_{j}\rvert^{p} \right)^{\frac{1}{p^2}}\left( r^{q}_{j} \fint_{B_{j}}  \lvert f \rvert^{q} \right)^{\frac{1}{q}} \\
	&\leq \varepsilon \left( \fint_{B_{j}} \lvert d u - d w_{j}\rvert^{p} \right)^{\frac{1}{p}} 
	+ c \left( r^{q}_{j} \fint_{B_{j}}  \lvert f \rvert^{q} \right)^{\frac{1}{q(p-1)}} \\
	&\stackrel{\eqref{f bound p less 2}}{\leq} \varepsilon \left( \fint_{B_{j}} \lvert d u - d w_{j}\rvert^{p} \right)^{\frac{1}{p}} 
	+ c\lambda^{2-p}r_{j}\left(  \fint_{B_{j}}  \lvert f \rvert^{q} \right)^{\frac{1}{q}}  .
	\end{align*}
	Similarly, for the second term, we deduce, 
	\begin{align*}
	I_{2} &\stackrel{\eqref{du bound p less 2}}{\leq}  
	\lambda^{\frac{2-p}{2}}\left( \fint_{B_{j}} \left\lvert V(du) - V(d w_{j}) \right\rvert^{2}\right)^{\frac{1}{2}} \\
	&\stackrel{\eqref{estimateuwforanyp}}{\leq} c \lambda^{\frac{2-p}{2}}\left( \fint_{B_{j}}  \lvert f \rvert \lvert u -w_{j}\rvert \right)^{\frac{1}{2}} \\
	&\leq c \lambda^{\frac{2-p}{2}} 
	\left( \fint_{B_{j}} \lvert d u - d w_{j}\rvert^{p} \right)^{\frac{1}{2p}}\left( r^{q}_{j} \fint_{B_{j}}  \lvert f \rvert^{q} \right)^{\frac{1}{2q}} \\
	&\leq \varepsilon \left( \fint_{B_{j}} \lvert d u - d w_{j}\rvert^{p} \right)^{\frac{1}{p}} 
	+ c\lambda^{2-p}r_{j}\left(  \fint_{B_{j}}  \lvert f \rvert^{q} \right)^{\frac{1}{q}} .
	\end{align*}
	Combining the last three estimates and choosing $\varepsilon > 0$ small enough to absorb the terms on the left hand side, we have shown \eqref{lambdasecondpsmaller2}. Now we show 
	\begin{align}\label{dv -dw psmaller2}
	\left( \fint_{\frac{1}{2}B_{j}} \left\lvert d w_{j} - d v_{j} \right\rvert^{p}\right)^{\frac{1}{p}} \leq c \omega(r_{j})\lambda 
	+ c\lambda^{2-p}r_{j}\left( \fint_{B_{j}} \left\lvert f \right\rvert^{q} \right)^{\frac{1}{q}} .
	\end{align}
	Note that \eqref{dv -dw psmaller2} combined with \eqref{lambdasecondpsmaller2} implies \eqref{lambdafirstpsmaller2d}. By \eqref{constant cv}, we have, 
	$$ \left\lvert dv_{j} - dw_{j}\right\rvert^{p} 
	\leq c \left\lvert V(dv_{j}) - V(dw_{j}) \right\rvert^{p}\left( \left\lvert dv_{j}\right\rvert +\left\lvert dw_{j}\right\rvert \right)^{\frac{(2-p)p}{2}}.$$
	Thus, by H\"{o}lder inequality, 
	\begin{align*}
	&\left( \fint_{\frac{1}{2}B_{j}} \left\lvert d w_{j} - d v_{j} \right\rvert^{p}\right)^{\frac{1}{p}} \\
	&\leq \left( \fint_{\frac{1}{2}B_{j}} \left\lvert V(dv_{j}) - V(d w_{j}) \right\rvert^{2}\right)^{\frac{1}{2}}
	\left( \fint_{\frac{1}{2}B_{j}} \left( \left\lvert d v_{j} \right\rvert + \left\lvert d w_{j} \right\rvert \right)^{p}\right)^{\frac{(2-p)}{2p}} \\
	&\stackrel{\eqref{l1omegageneral}}{\leq}  
	c \omega(r_{j}) \left( \fint_{\frac{1}{2}B_{j}}  \left\lvert d w_{j} \right\rvert^{p} \right)^{\frac{1}{p}}  \qquad \qquad \qquad \text{[ Using Proposition \ref{minimizerexistenceprop} ]}\\
	&\leq c \omega(r_{j}) \left( \fint_{\frac{1}{2}B_{j}}  \left\lvert d u \right\rvert^{p} \right)^{\frac{1}{p}} 
	+ c \omega(r_{j}) \left( \fint_{\frac{1}{2}B_{j}}  \left\lvert du - d w_{j} \right\rvert^{p} \right)^{\frac{1}{p}}  \\
	&\stackrel{\eqref{lambdasecondpsmaller2}\eqref{du bound p less 2}}{\leq} c \omega(r_{j})\lambda 
	+ c\lambda^{2-p}r_{j}\left( \fint_{B_{j}} \left\lvert f \right\rvert^{q} \right)^{\frac{1}{q}} .
	\end{align*}
	This concludes the proof.   \end{proof}

\subsection{Proof of theorem \ref{main theorem}}
\paragraph{Pointwise bounds}
\begin{theorem}\label{gradientsupboundgeneral}
	Let $u$ be as in theorem \ref{main theorem}. Then $d u$ is locally bounded in $\Omega.$ Moreover, there exists a constant 
	$c = c(n,k,p, \gamma, L, \omega( \cdot )) \geq 1$ and a positive radius $R_{0}= R_{0}(n,k,p, \gamma, L, \omega( \cdot )) >0$ such that the pointwise estimate 
	$$ \left\lvert d u (x_{0}) \right\rvert \leq c \left( \fint_{B(x_{0}, R)} \left\lvert d u \right\rvert^{s}\right)^{\frac{1}{s}} 
	+ c\left\lVert f \right\rVert^{\frac{1}{p-1}}_{L^{(n,1)}} $$
	holds whenever $B(x_{0},2R) \subset \Omega$, $2R \leq R_{1}$ and $x_{0}$ is a Lebesgue point of $d u.$ If $a(\cdot )$ is a constant function, the estimate holds without any restriction on 
	$R.$
\end{theorem}

\begin{proof}
	With lemma \ref{finalcomparispbigger2} and lemma \ref{finalcomparispsmaller2} at our disposal, now the arguments of the proof of Theorem 4 in \cite{KuusiMingione_nonlinearStein} works verbatim ( with the obvious notational modifications of writing $du,  dw_{j}, dw_{j-1}, dv_{j}$ in place of $Du, Dw_{j}, Dw_{j-1}, Dv_{j}$ etc ) to conclude the proof. We skip the details.  
	 \end{proof} 

\paragraph{Continuity of the exterior derivative}
Now the proof of continuity in theorem \ref{main theorem} follows exactly as in the proof of Theorem 1 in \cite{KuusiMingione_nonlinearStein}, with the obvious notational modifications mentioned above. We skip the details. \smallskip 

\paragraph{Proof of VMO regularity} 
The VMO regularity for the gradient now follows from estimates for $du$ by local estimates. Indeed, by local estimates \eqref{local estimate campanato}, we have, for any ball $B_{r} \subset \subset \Omega,$ 
\begin{align*}
\left[ \nabla u \right]_{\mathcal{L}^{2,n }(B(x,r/2))} \leq  c\left( \left[ du \right]_{\mathcal{L}^{2,n}(B(x,r))} + \left\lVert \nabla u \right\rVert_{L^{p}(B(x,r))} \right). 
\end{align*}
But since $du$ is continuous and $u \in W^{1,p}$, the right hand side can be made arbitrarily small by choosing $r$ small enough, proving the VMO regularity of $\nabla u.$

\section{Campanato estimates}
Consider the inhomogeneous quasilinear system 
\begin{align}\label{quasilinear maxwell tangential divergenceform}
\left\lbrace \begin{aligned}
\delta ( a(x) \lvert du \rvert^{p-2} du )  &= \delta F   &&\text{ in } \Omega,\\
\delta u &= 0 &&\text{ in } \Omega, 
\end{aligned} 
\right. 
\end{align}
where $u:\Omega \subset \mathbb{R}^{n} \rightarrow \Lambda^{k}\left( \mathbb{R}^{n}\right)$ for some $0 \leq k \leq n-1,$ $p \geq 2$ and the coefficient function
$a:\Omega \rightarrow [ \gamma, L ]$ is $C^{0,\alpha}_{\text{loc}}\left( \Omega \right)$, where $0 < \gamma < L < \infty$ and $F \in L_{loc}^{p'}(\Omega; \mathbb{R}^{\binom{n}{k}\times n}).$ Let $\beta_{2}$ be the exponent given in \eqref{Vdvscaling}. 

\begin{theorem}[Campanato estimates]\label{campanato estimates}
	Let $n \geq 2 ,$ $N \geq 1$ and $0 \leq k \leq n-1$ be integers and let $\Omega \subset \mathbb{R}^{n}$ be open. Suppose that 
	\begin{itemize}
		\item[(i)] $a:\Omega \rightarrow [ \nu, L ]$ is $C^{0,\alpha}_{\text{loc}}\left( \Omega \right)$, where $0 < \nu < L < \infty.$ 
		\item[(ii)]$F \in \mathcal{L}_{loc}^{p',\lambda}\left(\Omega; \Lambda^{k}\mathbb{R}^{n}\otimes \mathbb{R}^{N} \right)$ for some $0 \leq \lambda < \min \lbrace  n+ 2\alpha, n+2\beta_{2} \rbrace.$ 
	\end{itemize} 
	Let $ 2 \leq  p < \infty $ and let $u \in 
	W_{loc}^{1,p} \left(\Omega; \Lambda^{k}\mathbb{R}^{n}\otimes \mathbb{R}^{N} \right)$ be a local weak solution to \eqref{quasilinear maxwell tangential divergenceform}. Then we have $$ \nabla u \in \mathcal{L}_{loc}^{2, \frac{np -2n + 2\lambda}{p}} \left( \Omega; \mathbb{R}^{\tbinom{n}{k} \times n }\otimes \mathbb{R}^{N} \right) .$$ 
\end{theorem}
\subsection{Proof of the Campanato estimate}
\paragraph{General setting}
Let $x_{0} \in \Omega$ and $0 < R <1 $ be such that $B(x_{0}, 2R) \subset \Omega.$ By lemma \ref{gauge fixing lemma}, it is enough to consider the system
\begin{align}\label{quasilinear maxwell local coulomb campanato}
\left\lbrace \begin{aligned}
\delta ( a(x) \lvert du \rvert^{p-2} du) )  &= \delta F   &&\text{ in } B_{R},\\
\delta u &= 0 &&\text{ in } B_{R}.
\end{aligned} 
\right. 
\end{align}
Now we define $w \in u + W_{\delta,T}^{1,p}\left( B_{R}; \Lambda^{k}\mathbb{R}^{n}\otimes \mathbb{R}^{N}\right)$ to be the unique solution of 
\begin{align}\label{quasilinear maxwell tangential homogeneous general campanato}
\left\lbrace \begin{aligned}
\delta ( a(x) \lvert dw \rvert^{p-2} dw) )  &= 0   &&\text{ in } B_{R},\\
\delta w &= 0 &&\text{ in } B_{R}, \\
\nu\wedge w &= \nu\wedge u &&\text{  on } \partial B_{R},
\end{aligned} 
\right. 
\end{align}
and $v \in w + W_{\delta,T}^{1,p}\left( B_{R}; \Lambda^{k}\mathbb{R}^{n}\otimes \mathbb{R}^{N}\right)$ to be the unique solution of 
\begin{align}\label{quasilinear maxwell tangential homogeneous frozen general campanato}
\left\lbrace \begin{aligned}
\delta ( a(x_{0}) \lvert dv \rvert^{p-2} dv) )  &= 0   &&\text{ in }  B_{R},\\
\delta v &= 0 &&\text{ in } B_{R}, \\
\nu\wedge v &= \nu\wedge w_{j} &&\text{  on } \partial B_{R}.
\end{aligned} 
\right. 
\end{align}
\paragraph{Comparison estimates}
\begin{lemma}\label{firstcomparisoncampanato}
	Let $u$ be as in \eqref{quasilinear maxwell local coulomb campanato} and $w$ be as in \eqref{quasilinear maxwell tangential homogeneous general campanato}.  Then we have the following inequality
	\begin{equation}\label{firstcomparisonestimate campanato}
	\int_{B_{R}} \left\lvert V(du) - V(dw) \right\rvert^{2}  \leq c \int_{B_{R}} \left\lvert F - \left( F \right)_{B_{R}} \right\rvert^{\frac{p}{p-1}}.
	\end{equation}
\end{lemma}
\begin{proof}
	Weak formulation of \eqref{quasilinear maxwell local coulomb campanato} and \eqref{quasilinear maxwell tangential homogeneous general campanato} gives, 
	$$ \int_{B_{j}} a(x)\left\langle \lvert du \rvert^{p-2}du - \lvert dw \rvert^{p-2}dw; d\phi\right\rangle = \int_{B_{R}} \left\langle F ; d\phi\right\rangle = \int_{B_{R}} \left\langle F - \left(F\right)_{B_{R}}; d\phi\right\rangle , $$ 
	for all $ \phi \in W^{1,p}_{\delta,T}\left( B_{R};\Lambda^{k}\mathbb{R}^{n}\otimes \mathbb{R}^{N} \right). $
	Substituting $u -w$ in place of $\phi,$ we obtain, 
	$$ \int_{B_{R}} a(x)\left\langle \lvert du \rvert^{p-2}du - \lvert dw \rvert^{p-2}dw; du - dw\right\rangle = \int_{B_{R}} \left\langle F - \left(F\right)_{B_{R}}; du-dw\right\rangle . $$
	By classical monotonicity estimate \eqref{monotonicity}, we obtain, 
	\begin{align*}
	\int_{B_{R}} \left\lvert V(du) - V(dw) \right\rvert^{2}   &\leq \int_{B_{R}}\left( \left\lvert du \right\rvert + \left\lvert dw \right\rvert  \right)^{p-2}\left\lvert du  - dw\right\rvert^{2} \\ 
	&\leq c  \int_{B_{R}} a(x)\left\langle \lvert du \rvert^{p-2}du - \lvert dw \rvert^{p-2}dw; du - dw\right\rangle  \\&\leq \left\lvert \int_{B_{R}} \left\langle F - \left(F\right)_{B_{R}}; du-dw\right\rangle \right\rvert \\
	&\leq \int_{B_{R}} \left\lvert F - \left( F \right)_{B_{R}}\right\rvert \left\lvert du -dw  \right\rvert. 
	\end{align*}
	Now by Young's inequality with $\varepsilon > 0,$ we have, 
	\begin{align*}
	\int_{B_{R}} \left\lvert F - \left( F \right)_{B_{R}}\right\rvert &\left\lvert du -dw  \right\rvert \\&\leq \varepsilon \int_{B_{R}} \left\lvert d u - dw  \right\rvert^{p} + C_{\varepsilon} \int_{B_{R}} \left\lvert F - \left( F \right)_{B_{R}} \right\rvert^{\frac{p}{p-1}}\\
	&\leq c\varepsilon \int_{B_{R}} \left\lvert V(du) - V(dw) \right\rvert^{2}   + C{\varepsilon} \int_{B_{R}} \left\lvert F - \left( F \right)_{B_{R}} \right\rvert^{\frac{p}{p-1}}, 
	\end{align*}
	where in the last line we used the fact that $p \geq 2.$
	Choosing $\varepsilon>0$ small enough, we obtain the desired estimate. 
\end{proof}\smallskip 

\begin{lemma}\label{secondcomparison1campanato}
	Let  $w$ be as in \eqref{quasilinear maxwell tangential homogeneous general campanato} and $v$ be as in \eqref{quasilinear maxwell tangential homogeneous frozen general campanato}.  Then there exists a constant $c \equiv c\left( n, N, k, p, \gamma, L \right)$ such that we have the inequality
	\begin{align}\label{l1omega campanato}
	\fint_{B_{R}} \lvert V(dv) - V(dw) \rvert^{2}  &\leq c\left[ \omega\left(R\right) \right]^{2}  
	\fint_{B} \lvert V(dw) \rvert^{2}.
	\end{align}
\end{lemma}
\begin{proof} The proof follows by exactly the same arguments ( with the obvious notational changes ) as in lemma \ref{secondcomparison1}. 
\end{proof}\smallskip

\noindent Combining \eqref{firstcomparisonestimate campanato} with \eqref{l1omega campanato}, we have 
\begin{lemma}\label{secondcomparison}
	Let $p \geq 2$ and let  $u, v, w$ be as in \eqref{quasilinear maxwell local coulomb campanato}, \eqref{quasilinear maxwell tangential homogeneous general campanato} and \eqref{quasilinear maxwell tangential homogeneous frozen general campanato}, respectively. Then we have the following inequality
	\begin{align}\label{l1omegacombined}
	\int_{B_{R}} \lvert V(du) - V(dv) \rvert^{2}  &\leq c\left[ \omega\left(R\right) \right]^{2} \int_{B_{R}} \lvert V(du) \rvert^{2} + c \int_{B_{R}} \left\lvert F - \left( F \right)_{B_{R}} \right\rvert^{\frac{p}{p-1}}. \end{align}
\end{lemma}
Now we are ready to prove the Campanato estimate. \smallskip 
  
\begin{proof}{\textbf{of theorem \eqref{campanato estimates}}} We divide the proof in two steps. \smallskip 
	
	\noindent \emph{Step 1}
	First, we are going to prove that under the assumptions of the theorem, we have $V(du) \in \mathcal{L}_{loc}^{2,\lambda} \left(\Omega, \Lambda^{k}\mathbb{R}^{n}\otimes\mathbb{R}^{N} \right).$ Moreover, for any $\tilde{\Omega} \subset \subset \Omega,$ there exists a constant $c\equiv c(\tilde{\Omega }, n, N, k, p, \nu, L, \alpha ) > 0 $ such that the estimate
	\begin{equation}\label{Vdu campanato estimate}
	\left[ V(du) \right]_{\mathcal{L}^{2, \lambda}\left(\tilde{\Omega }; \Lambda^{k}\mathbb{R}^{n}\otimes\mathbb{R}^{N}\right)}^{2} \leq c \left( \lVert du \rVert_{L^{p}\left(\Omega ; \Lambda^{k}\mathbb{R}^{n}\otimes\mathbb{R}^{N}\right)}^{p} + \left[ F \right]_{\mathcal{L}^{p', \lambda}\left(\Omega ; \Lambda^{k}\mathbb{R}^{n}\otimes\mathbb{R}^{N}\right)}^{p'}\right) 
	\end{equation} holds. We first choose $x_{0} \in \Omega$ and radii $\rho, R >0$ such that $0 < 4\rho < R$ and $B_{R}(x_{0}) \subset \subset \Omega. $ Now we define the comparison functions $v$ and $w$ the same way as before in $B_{R}.$ We divide the proof in three substeps.\smallskip 
	
	\noindent \emph{Step 1a} We show the result for $0 < \lambda < n.$ Note that for $\lambda$ in this range, we have the identification $ \mathrm{L}^{2,\lambda} \simeq \mathcal{L}^{2,\lambda}.$ 
	We have, 
	\begin{align*}
	\int_{B_{\rho}} \lvert V(du) \rvert^{2} &\leq c   \left( \int_{B_{\rho}} \lvert V(du) - V(dv) \rvert^{2} +  \int_{B_{\rho}} \lvert V(dv) \rvert^{2} \right) \\
	&\stackrel{\eqref{l1omegacombined}}{\leq} c \left[ \omega\left(R\right) \right]^{2} \int_{B_{R}} \lvert V(du) \rvert^{2} + c \int_{B_{R}} \left\lvert F \right\rvert^{\frac{p}{p-1}} +  c \left( \rho \right)^{n}\sup\limits_{B_{\rho}}  \ \lvert dv \rvert^{p} \\
	&\stackrel{\eqref{supestimatefrozenhomogeneous}}{\leq} c \left[ \omega\left(R\right) \right]^{2} \int_{B_{R}} \lvert V(du) \rvert^{2} + c \int_{B_{R}} \left\lvert F \right\rvert^{\frac{p}{p-1}}  
	+  c \left( \frac{\rho}{R} \right)^{n} \int_{B_{R/2}} \lvert dv \rvert^{p} \\
	&\leq  c \left[\left( \frac{2\rho}{R} \right)^{n} + \left[ \omega\left(R\right) \right]^{2} \right] \int_{B_{R}} \lvert V(du) \rvert^{2} 
	+ c R^{\lambda}\left\lVert F \right\rVert^{p'}_{\mathrm{L}^{p', \lambda}}, 
	\end{align*}
	where in the last line we have used the fact that $v$ minimizes the functional $$v \mapsto \int_{B_{R}} a(x_{0})\left\lvert dv \right\rvert^{p} \quad  \text{ on } \quad u + W^{1,p}_{\delta, T} (B_{R}; \Lambda^{k}\mathbb{R}^{n}\otimes \mathbb{R}^{N}).$$
	Now choosing $R$ small enough and applying the standard iteration lemma (cf. Lemma 5.13 in \cite{giaquinta-martinazzi-regularity}) , we obtain the result for $0 < \lambda < n.$\smallskip 
	
	\noindent \emph{Step 1b} Now using  \eqref{Vdvscaling} and standard estimates,  we have, 
	\begin{align*}
	\int_{B_{\rho}} \lvert V(du)  &-\left( V(du) \right)_{B_{\rho}} \rvert^{2} \\ 
	&\leq  c \left( \frac{\rho}{R} \right)^{ n+ 2\beta_{1}} \int_{B_{R}} \lvert V(du) - \left( V(du) \right)_{B_{R}} \rvert^{2} + c  \int_{B_{R}} \lvert V(du)  - V(dv) \rvert^{2}  . \end{align*}
	Estimating the last integral with \eqref{l1omegacombined}, we get 
	the estimate 
	\begin{multline}\label{estimate for du}
	\int_{B_{\rho}} \lvert V(du)   -\left( V(du) \right)_{B_{\rho}} \rvert^{2}   \\ \leq c_{1}\left(\frac{\rho}{R} \right)^{n+ 2\beta_{1}}
	\int_{B_{R}} \lvert V(du) - \left( V(du) \right)_{B_{R}} \rvert^{2}  + c_{2} \left( \omega\left(R\right) \right)^{2}  
	\int_{B_{R}} \lvert V(du) \rvert^{2} \\ + c_{3} \int_{B_{R}} \lvert F - \left( F \right)_{B_{R}}\rvert^{p'} 
	\end{multline} 
	For $\lambda = n, $ since by step 1, $V(du) \in \mathrm{L}^{2, n-\varepsilon}(B_{R})$ for every $\varepsilon > 0,$ we choose $ \varepsilon > 0$  such that $ \lambda = n < n-\varepsilon + 2\alpha <  n + 2\beta $  and plugging the estimate 
	$$ c_{2} \left( \omega\left(R\right) \right)^{2}  
	\int_{B_{R}} \lvert V(du) \rvert^{2} \leq c R^{n -\varepsilon +2\alpha }$$ in \eqref{estimate for du}, we deduce  
	\begin{multline*}
	\int_{B_{\rho}} \lvert V(du)   -\left( V(du) \right)_{B_{\rho}} \rvert^{2}  \\ \leq c_{1}\left(\frac{\rho}{R} \right)^{n+ 2\beta_{1}}
	\int_{B_{R}} \lvert V(du) - \left( V(du) \right)_{B_{R}} \rvert^{2} 
	+ cR^{n-\varepsilon +2\alpha}+ \left[ F \right]_{\mathcal{L}^{p', \lambda}}^{p'} R^{\lambda}. 
	\end{multline*} 
	By the iteration lemma again (cf. Lemma 5.13 in \cite{giaquinta-martinazzi-regularity}), this proves $V(du) \in \mathcal{L}^{2,n}.$\smallskip  
	
	\noindent \emph{Step 1c} For $n < \lambda < n+2,$ we choose $\varepsilon > 0$ such that 
	$n < n-\varepsilon + 2\alpha <  \lambda .$ Then by the same arguments, $V(du) \in \mathcal{L}^{2, n-\varepsilon + 2\alpha}$ and thus $V(du) \in C^{0, \alpha - \frac{\varepsilon}{2}}_{loc}.$ In particular, $V(du)$ is bounded and thus, we have , 
	$$ c_{2} \left( \omega\left(R\right) \right)^{2}  
	\int_{B_{R}} \lvert V(du) \rvert^{2} \leq c \sup_{B_{R}}\left\lvert V(du) 
	\right\rvert^{2} R^{n +2\alpha }.$$ Now, plugging this back in \eqref{estimate for du} and using the iteration lemma again proves the estimate \eqref{Vdu campanato estimate} in the usual way.\smallskip 
	
	\noindent \emph{Step 2} Now for any ball $B_{\rho}$ compactly contained in $\Omega,$ we have, using the properties of mean, \eqref{constant cv} for $p \geq 2$ and the fact that $V$ is a bijection,   
	\begin{align*}
	\int_{B_{\rho}} \left\lvert du   -\left( du \right)_{B_{\rho}} \right\rvert^{p} &\leq 2 \int_{B_{\rho}} \left\lvert du   -V^{-1}\left( \left( V(du) \right)_{B_{\rho}}\right)  \right\rvert^{p}\\
	&\leq c \int_{B_{\rho}} \left\lvert V(du)   -\left( V(du) \right)_{B_{\rho}} \right\rvert^{2} \\
	&\stackrel{\eqref{Vdu campanato estimate}}{\leq} c \rho^{\lambda} \left( \lVert du \rVert_{L^{p}\left(\Omega ; \Lambda^{k}\mathbb{R}^{n}\otimes\mathbb{R}^{N}\right)}^{p} + \left[ F \right]_{\mathcal{L}^{p', \lambda}\left(\Omega ; \Lambda^{k}\mathbb{R}^{n}\otimes\mathbb{R}^{N}\right)}^{p'}\right). 
	\end{align*}
	This implies the estimate 
	\begin{align*}
	\int_{B_{\rho}} \left\lvert du   -\left( du \right)_{B_{\rho}} \right\rvert^{2} &\stackrel{\text{H\"{o}lder}}{\leq } c \rho^{\frac{n(p-2)}{p}}\left( \int_{B_{\rho}} \left\lvert du   -\left( du \right)_{B_{\rho}} \right\rvert^{p}\right)^{\frac{2}{p}}  \\
		&\leq c \rho^{\frac{np-2n + 2\lambda }{p}} \left( \lVert du \rVert_{L^{p}\left(\Omega ; \Lambda^{k}\mathbb{R}^{n}\otimes\mathbb{R}^{N}\right)}^{p} + \left[ F \right]_{\mathcal{L}^{p', \lambda}\left(\Omega ; \Lambda^{k}\mathbb{R}^{n}\otimes\mathbb{R}^{N}\right)}^{p'}\right)^{\frac{2}{p}}. 
	\end{align*}
	This implies $du \in \mathcal{L}_{loc}^{2,\frac{np-2n + 2\lambda }{p}}\left(\Omega, \Lambda^{k}\mathbb{R}^{n}\otimes\mathbb{R}^{N} \right)$ along with the corresponding estimates. By virtue of \eqref{local estimate campanato}, this implies that for any $B_{r} \subset \subset \Omega,$ we have  the estimate 
	\begin{multline}\label{campanato for grad u}
	\left[ \nabla u \right]_{\mathcal{L}^{2, \frac{np-2n + 2\lambda}{p}} \left( B_{r/2}; \Lambda^{k}\mathbb{R}^{n}\otimes\mathbb{R}^{N}\right)}^{2} \\ \leq c \left( \lVert du \rVert_{L^{p}\left(B_{r} ; \Lambda^{k}\mathbb{R}^{n}\otimes\mathbb{R}^{N}\right)}^{p} + \left[ F \right]_{\mathcal{L}^{p', \lambda}\left(B_{r} ; \Lambda^{k}\mathbb{R}^{n}\otimes\mathbb{R}^{N}\right)}^{p'}\right)^{\frac{2}{p}}.
	\end{multline} 
\end{proof} 
\subsection{Proof of Theorem \ref{holder theorem}}
\begin{proof}[Proof of theorem \eqref{holder theorem}]
	Given any ball $B_{R} \subset \subset \Omega$ and given $f,$ we first solve the system 
	\begin{align}\label{changing into divergenceform}
	\left\lbrace \begin{aligned}
	dF &= 0 &&\text{ in } B_{R} \\
	\delta F &= f &&\text{ in } B_{R} \\
	\nu\wedge F &= 0 &&\text{ on } \partial B_{R}.
	\end{aligned} \right. 
	\end{align} 
	Thus, along with lemma \ref{gauge fixing lemma}, local estimates for the system \eqref{main equation holder} boils down to proving estimate for the system \eqref{quasilinear maxwell local coulomb campanato}. Also, linear estimates for the system \eqref{changing into divergenceform} implies the following 
	\begin{itemize}
		\item If $f \in L^{q}$ with $q > n,$ then $F \in W^{1,q} \hookrightarrow C^{0, \frac{q-n}{q}} \simeq \mathcal{L}^{p', n + \frac{p'(q-n)}{q}}.$
		\item If $f \in L^{n},$ then $F \in W^{1,n} \hookrightarrow VMO,$ which implies that for any ball $B_{r} \subset \subset B_{R},$ the BMO seminorm $\left[F\right]_{\mathcal{L}^{p', n}\left( B_{r}; \Lambda^{k+1}\mathbb{R}^{n}\otimes\mathbb{R}^{N} \right)} \rightarrow 0$ as $r \rightarrow 0. $ 
		\item If $f \in L^{(n, \infty )},$ then $F \in W^{1,(n, \infty)} \hookrightarrow BMO.$
	\end{itemize} 
In the first and the third case, the conclusion of theorem \eqref{holder theorem}
follows immediately from these embeddings and theorem \ref{campanato estimates}. For the second case, we note that by the estimate \eqref{campanato for grad u} for $\lambda = n$, we have, for any ball $B(x,r) \subset \subset B_{R} \subset \subset \Omega,$ 
\begin{align*}
\left[ \nabla u \right]_{\mathcal{L}^{2,n }(B(x,r/2))} \leq  c \left( \lVert du \rVert_{L^{p}\left( B(x,r) \right)}^{p} + \left[ F \right]_{\mathcal{L}^{p', \lambda}\left( B(x,r) \right)} \right)^{\frac{2}{p}}. 
\end{align*} Since $du \in L^{p}_{loc}\left(\Omega; \Lambda^{k}\mathbb{R}^{n}\otimes\mathbb{R}^{N} \right)$ and $F$ is VMO, the right hand side of the above estimate can be made arbitrarily small by choosing $r$ small enough. This concludes the proof.  \end{proof}

\section*{Acknowledgements}
The author thanks Giuseppe Mingione for attracting his attention to the question of continuity of the exterior derivative and the enlightening discussions. The present work and the author also drew much of its inspiration from the work of Mingione and co-authors for the $p$-laplacian system. The author warmly thanks Jan Kristensen for helpful and encouraging discussions.

\end{document}